\documentclass[11pt,leqno]{amsart}
\usepackage{amsmath,amsthm,amssymb,mathrsfs,stmaryrd,color, thmtools}
\usepackage[all]{xy}
\usepackage{url}
\usepackage{geometry}
\usepackage{setspace}
\usepackage[utf8]{inputenc}
\usepackage[T1]{fontenc}
\usepackage{braket}
\usepackage{graphicx}
\usepackage{mathtools}
\usepackage{relsize}
\usepackage[bbgreekl]{mathbbol}
\usepackage{amsfonts}
\usepackage{quiver} 
\usepackage{comment}
\DeclareSymbolFontAlphabet{\mathbb}{AMSb}
\DeclareSymbolFontAlphabet{\mathbbl}{bbold}

\usepackage{hyperref}
\hypersetup{
    colorlinks=true,
    citecolor=blue,
    linkcolor=blue,
    pagebackref=false,
    pdfpagelabels,
    hyperindex
}

\usepackage{enumitem}
\setlist[enumerate]{itemsep=2pt,parsep=2pt,before={\parskip=2pt}}

\usepackage{cleveref}

\makeatletter
\def\step{%
   \@ifnextchar[ \@step{\@noitemargtrue\@step[\@itemlabel]}}
\def\@step[#1]{\item[#1]\mbox{}\\\hspace*{\dimexpr-\labelwidth-\labelsep}}
\makeatother

\newcommand{\cosimp}[3]{\xymatrix@1{#1 \ar@<.4ex>[r] \ar@<-.4ex>[r] & {\ }#2 \ar@<0.8ex>[r] \ar[r] \ar@<-.8ex>[r] & {\ } #3 \ar@<1.2ex>[r] \ar@<.4ex>[r] \ar@<-.4ex>[r] \ar@<-1.2ex>[r] & \cdots }}

\newcommand{\adjunction}[4]{\xymatrix@1{#1{\ } \ar@<0.3ex>[r]^{ {\scriptstyle #2}} & {\ } #3 \ar@<0.3ex>[l]^{ {\scriptstyle #4}}}}

\DeclareMathOperator{\Q}{\mathbb{Q}}

\DeclareMathOperator{\F}{\mathbb{F}}
\DeclareMathOperator{\Z}{\mathbb{Z}}
\DeclareMathOperator{\GG}{\mathbf{G}}

\DeclareMathOperator{\C}{\mathbb{C}}

\newcommand{\x}{\times}

\newcommand{\spf}{\mathrm{Spf} }

\DeclareMathOperator{\SL}{\mathrm{SL}}
\DeclareMathOperator{\GL}{\mathrm{GL}}

\DeclareMathOperator{\Spf}{\mathrm{Spf}}

\DeclareMathOperator{\Hom}{\mathrm{Hom}}

\DeclareMathOperator{\Ext}{\mathrm{Ext}}

\newcommand{\ur}{\mathrm{ur}}

\DeclareMathOperator{\cX}{\mathcal{X}}

\DeclareMathOperator{\cM}{\mathcal{M}}
\DeclareMathOperator{\cO}{\mathcal{O}}

\DeclareMathOperator{\cW}{\mathcal{W}}
\DeclareMathOperator{\cP}{\mathcal{P}}

\DeclareMathOperator{\cR}{\mathcal{R}}

\DeclareMathOperator{\sX}{\mathscr{X}}

\DeclareMathOperator{\et}{\mbox{\scriptsize \'et}}
\DeclareMathOperator{\Gal}{Gal}

\DeclareMathOperator{\opt}{\scriptscriptstyle \mathrm{opt}}

\DeclareMathOperator{\Spec}{\mathrm{Spec}}
\DeclareMathOperator{\Spa}{\mathrm{Spa}}
\DeclareMathOperator{\Spd}{\mathrm{Spd}}

\DeclareMathOperator{\rig}{\mathrm{rig}}

\newcommand{\RZ}{\mathrm{RZ}}

\newcommand{\ad}{\mathrm{ad}}
\renewcommand{\H}{\mathrm{H}}
\newcommand{\der}{\mathrm{der}}
\newcommand{\PGL}{\mathrm{PGL}}
\newcommand{\cG}{\mathcal{G}}

\newtheorem{theorem}{Theorem}[section]
\newtheorem*{theorem*}{Theorem}
\newtheorem*{definition*}{Definition}
\newtheorem{proposition}[theorem]{Proposition}
\newtheorem{lemma}[theorem]{Lemma}
\newtheorem{corollary}[theorem]{Corollary}

\theoremstyle{definition}
\newtheorem{definition}[theorem]{Definition}

\newtheorem{remark}[theorem]{Remark}

\newtheorem{example}[theorem]{Example}

\newtheorem{convention}[theorem]{Convention}

\newtheorem{conjecture}[theorem]{Conjecture}
\crefname{assumption}{assumption}{assumptions}
\crefname{construction}{construction}{constructions}

\begin{document}

\title[]{On the Harris--Viehmann conjecture for Hodge--Newton reducible local Shimura data of abelian type}

\author[]{Sandra Nair, Xinyu Zhou}

\maketitle
\setcounter{tocdepth}{1}

\begin{abstract}
We address a new case of the Harris--Viehmann conjecture, which establishes a parabolic induction formula on the cohomology groups associated to non-basic local Shimura data. It follows that all supercuspidal representations on a Shimura variety are concentrated along the basic locus, making the conjecture relevant to the Langlands program. Historically, many cases of the Harris--Viehmann conjecture have been approached with the additional condition of Hodge--Newton reducibility on the underlying local Shimura datum. Building on previous work by Mantovan (EL/PEL case) and Hong (Hodge case), we extend the proof of the conjecture to unramified non-basic local Shimura data of abelian type under the assumption of Hodge--Newton reducibility. We leverage Shen’s construction of Rapoport--Zink spaces of abelian type at the hyperspecial level.
\end{abstract}

{\hypersetup{linkcolor=black}\tableofcontents}

\section{Introduction} \label{Sec: 1}

  \emph{Rapoport--Zink spaces} are local analogues of integral models of Shimura varieties. The first example of a Rapoport--Zink space was explicitly constructed in \cite{rapoport2016period}, serving as a formal moduli space of $p$-divisible groups with PEL structure. Its $\ell$-adic cohomology realizes the local Langlands correspondence and the local Jacquet-Langlands correspondence, c.f. \cite{harrisgeometry} for $\GL_n$ case. This indicates a more general phenomenon, where the $\ell$-adic cohomology of appropriate local analogues of Shimura varieties can materialize the local Langlands correspondence and the local Jacquet-Langlands correspondence. Following this, Rapoport and Viehmann conjectured the existence of \emph{local Shimura varieties} in \cite{RapoportViehmann}, of which (the generic fibers of) Rapoport--Zink spaces are a special case. They start with a \emph{local Shimura datum}, and attach a tower of rigid analytic spaces to it. 

The \emph{Harris--Viehmann conjecture} concerns the cohomology of local Shimura varieties with underlying non-basic local Shimura data. Originally conjectured by Harris in \cite{harris2001local}, later modified by Viehmann in \cite{RapoportViehmann}, and then further amended by Bertoloni Meli in \cite{meli2022cohomology}, the conjecture gives a parabolic inductive formula for the $\ell$-adic cohomology of such local Shimura varieties. In particular, it asserts that all the supercuspidal representations are concentrated along the basic locus. 

\subsection{Main results}
To state our main result, we fix some notations first. We adopt the following conventions throughout the paper: 
\begin{itemize}
    \item Let $p$ be a fixed odd prime. Let $\ell$ be a prime distinct from $p$.
    \item Fix algebraic closures $\overline{\F}_p$ and $\overline{\Q}_p$ of $\F_p$ and $\Q_p$, respectively,
    \item $\Q_p^{\ur}$ is the maximal unramified extension of $\Q_p$,
    \item $\breve{\Q}_p$ and $\C_p$ are the $p$-adic completions of $\Q_p^{\ur}$ and $\overline{\Q}_p$ respectively,
    \item $W = W(\overline{\F}_p)$ is the ring of Witt vectors,
    \item $G$ is a connected reductive group over $\Q_p$,
    \item Let $\sigma \in \Gal(\breve{\Q}_p/ \Q_p)$ be the Frobenius automorphism. $b$ and $b'$ in $G(\breve{\Q}_p)$ are $\sigma$-conjugate if there exists some $g \in G(\breve{\Q}_p)$ such that $b' = gb\sigma(g)^{-1}$,
    \item $B(G)$ is the Kottwitz set associated to $G$ (c.f. \cite{kottwitz1985isocrystals}),
    \item $\Gamma = \mathrm{Gal}(\overline{\Q}_p/\Q_p)$ 
    \item For any $p$-adic field $F$, $\cW_F$ is the Weil group of $F$. \\
\end{itemize} 

Let $G$ be a connected reductive group over $\Q_p$. Let $(G, [b], \{\mu\})$ be an \emph{unramified local Shimura datum of abelian type} (c.f. \Cref{Sec: 3} for details). To this triple, we associate a \emph{Rapoport--Zink space of abelian type} $\RZ_G$ following Shen's construction in \cite{ShenAbelianRZ}. This is a formal scheme over $\spf(\breve{\Z}_p)$, with a rigid analytic fibre $\RZ_G^{\rig}$ and a tower of \'etale covers $\{\RZ_G^{K_p}\}$, where each $K_p \subset G(\Z_p)$ is an open compact subgroup. For each $i > 0$, these give rise to $\ell$-adic cohomology groups:

$$\H^i(\RZ_G^{K_p}) := \H_c^i(\RZ_G^{K_p} \otimes_{\breve{\Q}_p}\C_p, \Q_{\ell}(\dim \RZ_G^{K_p}))$$

In particular, the tower $\{\H^i(\RZ_G^{K_p})\}$ carries a natural action of $G(\Q_p) \times \cW_E \times J_b(\Q_p)$, where $\cW_E$ is the Weil group of the reflex field and $J_b$ is the algebraic group of inner automorphisms associated to the local Shimura datum (c.f. \Cref{Sec: 3}). Let $\rho$ be an admissible $\ell$-adic representation of $J_b(\Q_p)$. We first define
\[\H^{i,j}(\RZ_{G}^\infty)_\rho := \varinjlim_{K_p}\;\Ext^{j}_{J_b(\Q_p)}(H^i(\RZ_{G}^{K_p}),\rho)\]
From this, we get a virtual representation:
\[\mathrm{H}^\bullet(\RZ^\infty_{G})_\rho := \sum_{i,j \geq 0}(-1)^{i+j}\H^{i,j}(\RZ_{G}^\infty)_\rho \]

An important assumption we make is that the local Shimura datum is \emph{Hodge--Newton reducible} with respect to a fixed parabolic $P$ and a Levi $L$ of $G$. By definition, there is a particular choice of representatives $b \in [b] \cap L(\breve{\Q}_p)$ and $\mu \in \{\mu\}$ factoring through $L$ yielding an unramified local Shimura datum $(L, [b], \{\mu\})$ of abelian type (c.f. \Cref{Sec: 4} and \Cref{Sec: 7} for details). This assumption allows us to reduce $\ell$-adic cohomologies of the tower of rigid analytic spaces of $\RZ_G$ in a natural way.

Furthermore, for the main result, we will be working with a finite extension $E^{\opt}$ of the reflex field $E$ of the unramified local Shimura datum $(G, [b], \{\mu\})$ of abelian type, which only depends on the datum. For details on this construction, please refer to \Cref{thm: wdd-Shen}. Finally, we assume that $(G, [b], \{\mu\})$ has an associated Hodge lift $(H, [b'], \{\mu'\})$ such that the map $\pi_1(H)^{\Gamma} \to \pi_1(H^{\ad})^{\Gamma}$ is surjective (c.f. \Cref{Sec: 5} for further details).

Our main theorem is as follows:

\begin{theorem}\label{thm: HV full abelian}
Let $(G, [b], \{\mu\})$ be a local Shimura datum satisfying the assumptions above. Then for any admissible $\overline{\Q}_\ell$-representation $\rho$ of $J_{b}(\Q_p)$, there is an equality
    \begin{align*}
    \H^\bullet(\RZ_{G}^\infty)_\rho = \mathrm{Ind}_{P(\Q_p)}^{G(\Q_p)} \H^\bullet(\RZ_{L}^\infty)_\rho.
\end{align*}
as virtual representations of $G(\Q_p) \x \cW_{E^{\opt}}$.
\end{theorem}

\subsection{History of results on the Harris--Viehmann conjecture:}
\begin{enumerate}
    \item Boyer proved the Harris--Viehmann conjecture for Drinfeld modular varieties in \cite{boyer1999mauvaise}.
    \item Mantovan proved the existence of the canonical Hodge--Newton filtration for $p$-divisible groups with additional structures in \cite{mantovan2008non}. She passed to their Rapoport–Zink spaces of PEL type under the assumption that the Newton polygon coincides with the Hodge polygon up to or from the nontrivial break contact point, to prove the Harris--Viehmann conjecture.
    \item Shen followed up on this in \cite{shen2014hodge} by considering the more general case when the Newton polygon admits a nontrivial contact point with the Hodge polygon, and assuming that this contact point is a break point of the Newton polygon. This also established the Harris--Viehmann conjecture for Rapoport--Zink spaces of PEL type.
    \item Hansen gave another proof of the conjecture in \cite{hansen2021moduli} for $G = \mathrm{GL}_n$ under the Hodge--Newton reducibility assumption. He used the theory of local Shimura varieties developed by Scholze in \cite{scholze2018p}.
    \item Gaisin and Imai proved the conjecture (assuming Hodge--Newton reducibility) for a generalization of the diamond of a non-basic Rapoport--Zink space at infinite level in \cite{Gaisin_Imai_2025}.
    \item Hamann and Imai proved the conjecture under a mild Hodge--Newton reducibility condition for a general local Shimura variety, by studying the moduli stack of parabolic bundles on the Fargues--Fontaine curve in \cite{hamann2025dualizing}.
\end{enumerate}

Our main goal is to establish the Harris--Viehmann conjecture for unramified local Shimura data of abelian type, which are Hodge--Newton reducible with respect to fixed Levi and parabolic subgroups. We work with Rapoport--Zink spaces arising from unramified local Shimura data of abelian type as constructed in \cite{ShenAbelianRZ}. \\

\subsection{Road map for the paper}
In \Cref{Sec: 3}, we review some background, including relevant definitions. In \Cref{Sec: 4}, we recall the construction of Rapoport--Zink spaces of Hodge type in \cite{kim2018rapoport}, and summarize the Mantovan--Hong method, which proves the Harris--Viehmann conjecture for Hodge--Newton reducible unramified local Shimura data of Hodge type. In \Cref{Sec: 5}, we recount the construction of Rapoport--Zink spaces of abelian type in \cite{ShenAbelianRZ}, as well as some geometric properties that follow. We take a detour to address the Weil descent datum for local Shimura varieties of abelian type in \Cref{Sec: 6}, for completeness. \Cref{Sec: 7}-\Cref{Sec: 9} relates the geometry of Rapoport--Zink spaces of abelian type with Hodge--Newton reducibility, setting up for studying their cohomologies in \Cref{Sec: 10}. We put the results together to finish the proof of the Harris--Viehmann conjecture for Hodge--Newton reducible unramified local Shimura data of abelian type in \Cref{Sec: 10}.

\section{Acknowledgments} \label{Sec: 2}

We would like to thank Serin Hong for giving us valuable feedback on an early version of the draft. We also want to thank Xu Shen, Jeff Achter, Linus Hamann, and Alexander Bertoloni Meli for helpful conversations. 

\section{Background}\label{Sec: 3}

We first recall some important definitions.

\begin{definition} [{\cite[Definition 5.1]{RapoportViehmann}}]
    A local Shimura datum over $\Q_p$ is a triple $(G, [b], \{\mu\})$, where:
    \begin{enumerate}
        \item $G$ is a connected reductive group over $\Q_p$.
        \item $[b] \in B(G)$ is a $\sigma$-conjugacy class.
        \item $\{\mu\}$ is a conjugacy class of cocharacters $\mu: \mathbb{G}_m \to G_{\overline{\Q}_p}$
    \end{enumerate}
     such that the following conditions hold:
     \begin{enumerate}
         \item $[b] \in B(G, \{\mu\})$ (the subset of the Kottwitz set bounded by $\{\mu\}$, c.f. \cite{rapoport1996period}).
         \item $\mu$ is minuscule, i.e. $|\mu| \in \{-1,0,1\}$.
     \end{enumerate}
     Moreover, we have the following quantities associated with a local Shimura datum:
     \begin{enumerate}
         \item The field of definition $E=G_E$ of $\{\mu\}$ is called the reflex field, which is a finite extension of $\Q_p$. Let $\cW_{E}$ denote the Weil group of $E$. 
         \item The reductive algebraic group $J_b$ over $\Q_p$ for $b \in [b]$.
         It has the following functor of points:
         \begin{align*}
             J_b(R) = \{g \in G(R \otimes_R \breve{\Q}_p) \;:\; b = gb\sigma(g)^{-1} \}
         \end{align*}
         for any $\Q_p$-algebra $R$. Note that the isomorphism class of $J_b$ depends only on the $\sigma$-conjugacy class of $b$. $J_b$ is an inner form of a particular Levi subgroup of $G$ within which $b$ is a central element, c.f. \cite[Remark 1.15]{rapoport1996period}. 
       
     \end{enumerate}
\end{definition}

In \cite{rapoport1996period}, the local Shimura data considered came exclusively from PEL data. Like with the global theory of Shimura varieties, there are more general cases than that:

\begin{definition}[{\cite[Remark 5.4(i)]{RapoportViehmann}}]

    A local Shimura datum $(G, [b], \{\mu\})$ is of \emph{Hodge type} if there exists an embedding $G \hookrightarrow \GL(V)$ for some $\Q_p$-vector space $V$, and a local Shimura datum $(\GL(V), [b'], \{\mu'\})$ such that $[b]$ and $\{\mu\}$ are mapped to $[b']$ and $\{\mu'\}$ respectively. Here, $\{\mu'\}$ corresponds to $(1^r, 0^{n-r})$ for some integer $ 0 \leq r \leq n = \dim(V)$. 
\end{definition}

\begin{enumerate}
    \item Any local Shimura datum $(G, [b], \{\mu\})$ of EL/PEL type (coming from an EL/PEL Shimura datum) is also of Hodge type.
    \item Let $(G,X)$ be a (global) Shimura datum of Hodge type, i.e. there exists some embedding into the Siegel Shimura datum $(G,X) \hookrightarrow (\mathrm{GSp}, S^\pm)$, where $S^{\pm}$ are the connected components of the Siegel upper half-space. Let $\mu$ be the cocharacter associated to $X$, and $[b] \in B(G, \{\mu\})$. Then $(G_{\Q_p}, [b], \{\mu\} )$ is a local Shimura datum of Hodge type that is not of PEL type. 
\end{enumerate}

\begin{definition}\cite[Definition 3.4]{ShenAbelianRZ}
    The group $G$ is unramified if it is quasisplit over $\Q_p$ and splits over $\Q_p^{\mathrm{ur}}$. Equivalently, $G$ admits a reductive model over $\Z_p$. A local Shimura datum $(G, [b], \{\mu\})$ of Hodge type is \emph{unramified} if $G$ is unramified.
\end{definition}

\begin{convention}\label{conv: LSD}
    When additional clarification is required, we will write the local Shimura datum as $(G,[b_G], \{\mu_G\})$.
\end{convention}

Associated to a local Shimura datum $(G,[b_G], \{\mu_G\})$ is an ``adjoint datum'' $(G^{\ad},[b_{G^{\ad}}],\{\mu_{G^{\ad}}\})$, where $G^{\ad}:=G/Z(G)$ is the adjoint quotient of $G$. In fact, the quotient map $G\twoheadrightarrow G^{\ad}$ induces a map of Kottwitz sets $B(G)\to B(G^{\ad})$. We write $b_{G^{\ad}}$ for the image of $b_G\in B(G)$ under this map. We then write
\begin{equation*}
    \mu_{G^{\ad}}:\mathbb{G}_m\xrightarrow{\mu_G} G_{\overline{\Q}_p}\to G_{\overline{\Q}_p}^{\ad }
\end{equation*}
for the induced cocharacter of $G^{\ad}$. It is easy to check the triple $(G^{\ad}, [b_{G^{\ad}}], \{\mu_{G^{\ad}}\})$ indeed forms a local Shimura datum. The morphism $(G, [b_G], \{\mu_G\}) \to (G^{\ad}, [b_{G^{\ad}}], \{\mu_{G^{\ad}}\})$ is a morphism of local Shimura data. This leads us to the following definition:

\begin{definition}
    A local Shimura datum $(G, [b_G], \{\mu_G\})$ is called of \emph{abelian type}, if there exists a local Shimura datum of Hodge type $(H, [b_H], \{\mu_H\})$ such that there is an isomorphism of the associated adjoint local Shimura data $(G^{\ad}, [b_{G^{\ad}}], \{\mu_{G^{\ad}}\}) \cong (H^{\ad}, [b_{H^{\ad}}],\{\mu_{H^{\ad}}\})$. A local Shimura datum $(G, [b_G], \{\mu_G\})$ of abelian type is called \textit{unramified} if $G$ is unramified and there is an unramified local Shimura datum $(H, [b_H], \{\mu_H\})$ of Hodge type with $(G^{\ad}, [b_{G^{\ad}}], \{\mu_{G^{\ad}}\}) \cong (H^{\ad}, [b_{H^{\ad}}], \{\mu_{H^{\ad}}\})$.
\end{definition}

In particular, any local Shimura datum of Hodge type is also of abelian type, and the following example establishes the latter as a strictly larger class of objects:

\begin{example} \cite[Example 5.6]{RapoportViehmann}
    Let $G=\PGL_n$ for $n \geq 2$. Consider a nontrivial minuscule cocharacter $\mu_H:\GG_m\to \GL_n$ and $[b_H]\in B(\GL_n,\mu_H)$. Take $\mu_G=\mu_{H^{\ad}}$ and $[b_G]=[b_{H^{\ad}}]$. Then $(G, [b], \{\mu\})$ is an unramified local Shimura datum of (strictly) abelian type.
\end{example}

The associated local reflex field $E = E(G, [b_G], \{\mu_G\})$ is an unramified extension of $\Q_p$, giving rise to $\breve{E} = W_{\Q_p}$ and $\mathcal{O}_{\breve{E}} = W$.

From here onward, we assume that all the local Shimura data mentioned are unramified unless otherwise stated.

\section{The Hodge case}\label{Sec: 4}
\subsection{Rapoport--Zink spaces of Hodge type}
Fix an unramified local Shimura datum of Hodge type $(G, [b], \{\mu\})$ and a faithful $G$-representation $\Lambda \in \mathrm{Rep}_{\Z_p}(G)$. The condition of being unramified allows us to pick $b \in [b]$ and $\mu \in \{\mu\}$ such that $b \in G(\breve{\Z}_p)\mu(p)G(\breve{\Z}_p)$. This allows for the construction of an $F$-isocrystal with $G$-structure, that gives rise to a $p$-divisible group with $G$-structure up to isogeny, denoted by $\underline{X}=(X, (t_i))$ (c.f. \cite{kim2018rapoport} for the detailed construction). Let $\mathrm{Nilp}_{\breve{\Z}_p}$ denote the category of $\breve{\Z}_p$-algebras where $p$ is nilpotent. 

We define the set-valued covariant functor $\RZ_G$ on $\mathrm{Nilp}_{\breve{\Z}_p}$ as follows: for any $R \in \mathrm{Nilp}_{\breve{\Z}_p}$, $\RZ_G(R)$ is the set of isomorphism classes of pairs $(\mathcal{X}, \iota)$, where $\mathcal{X}$ is a $p$-divisible group and $\iota: X_{R/p} \to \mathcal{X}_{R/p}$ is a quasi-isogeny. $\RZ_G$ is independent of the choice of $b \in [b] \cap G(\breve{\Z}_p)\mu(p)G(\breve{\Z}_p)$ up to isomorphism. It was shown in \cite{rapoport1996period} that the functor $\RZ_G$ is represented by a formal scheme which is locally formally of finite type and formally smooth over $\breve{\Z}_p$. For brevity, we also write $\RZ_G$ for the formal scheme representing the functor and $\mathcal{X}$ for the universal $p$-divisible group over it.

For a pair $(\mathcal{X}, \iota)$, we get an isomorphism of $F$-isocrystals with $G$-structure induced by $\iota$:
\begin{align*}
    \mathbb{D}(\iota): \mathbb{D}(\mathcal{X}_{R/p})[1/p] \xrightarrow{\cong}  \mathbb{D}(X_{R/p})[1/p] 
\end{align*}

Let $(t_{\mathcal{X},i})$ denote the inverse image of the tensors $(t_i)_R$ under this isomorphism. 

Let $\mathrm{Nilp}_{\breve{\Z}_p}^{\mathrm{sm}}$ be the collection of formally smooth and formally finitely generated algebras over $\breve{\Z}_p/p^m$ for some $m \in \Z_{>0}$. This is a full subcategory of $\mathrm{Nilp}_{\breve{\Z}_p}$. 

We define the set-valued covariant functor $\RZ_{G,b}^{(s_i)}$ on $\mathrm{Nilp}_{\breve{\Z}_p}^{\mathrm{sm}}$ as follows: take any $R \in \mathrm{Nilp}_{\breve{\Z}_p}^{\mathrm{sm}}$, a morphism $f: \spf (R)\to \RZ_G$, and a $p$-divisible group $\mathcal{X}$ over $\Spec(R)$ which pulls back to $f^*\mathcal{X}_{\GL,b}$ over $\spf(R)$. Then $f \in \RZ_{G,b}^{(s_i)}(R)$ if and only if there exists a
(unique) family of tensors $(\textbf{t}_i)$ on $\mathbb{D}(\mathcal{X})$ such that the following holds:
\begin{itemize}
    \item For an ideal of definition $I$ of $R$ containing $p$, the pullback of $(\textbf{t}_i)$ over $R/I$ is compatible with the pullback of $(t_{\mathcal{X},i})$ over $R/I$.
    \item Let $\mathcal{R}$ be a $p$-adic lift of $R$ that is formally smooth over $\breve{\Z}_p$. Then the following $\mathcal{R}$-scheme is a $G$-torsor:
    \begin{equation*}
        \mathcal{P}_{\mathcal{R}}:=\mathrm{Isom}_{\mathcal{R}} ((\mathbb{D}(\mathcal{X})_ {\cR},(\textbf{t}_i)_{\cR}), (\Lambda^*\otimes_{\Z_p} \cR, s_i\otimes 1))
    \end{equation*}
    \item The Hodge filtration of $ \cX$ is a $\{\sigma^{-1}(\mu^{-1})\}$-filtration with respect to $\textbf{t}_i$.
\end{itemize}

By \cite[Theorem 4.9.1]{kim2018rapoport}, there exists a closed formal subscheme $\RZ_{G,b} \subset \RZ_{b}$ which is formally smooth over $\breve{\Z}_p$ and locally formally of finite type over $\spf(\breve{\Z}_p)$, representing the functor $\RZ_{G,b}^{(s_i)}$ for any choice of $(s_i)$. Its isomorphism class only depends on the local Shimura datum of Hodge type $(G, [b], \{\mu\})$. By \cite{berthelot1996cohomologie}, $\RZ_{G,b}$ admits a rigid analytic fibre, denoted by $\RZ_{G,b}^{\mathrm{rig}}$. By varying the level structure $K_p \subset G(\Z_p)$, we can construct a tower of \'etale coverings of this rigid analytic fibre, written as $\RZ_{G, b}^{K_p}$. At the infinite level, this construction gives rise to $\RZ_{G,b}^\infty$, which is the local Shimura variety of Hodge type conjectured by \cite{RapoportViehmann} and proven in \cite{kim2018rapoport}. The cohomology groups of this object detect the representations of $G(\Q_p)$, $J_b(\Q_p)$, and $\cW_E$.

\subsection{Hodge--Newton reducibility for the Hodge type}

\begin{definition}[{\cite[\S 4.1.4]{HongPublished}}]\label{def:HN-reducible}
    Let $(G,[b],\{\mu\})$ be an unramified local Shimura datum. Fix a parabolic subgroup $P$ of $G$, with Levi $L$ and unipotent radical $U$. The datum $(G,[b],\{\mu\})$ is called \textit{Hodge--Newton reducible} with respect to $P$ and $L$ if there exists $\mu\in\{\mu\}$ and $b\in [b]\cap L(\breve{\Q}_p)$ with the following properties:
    \begin{enumerate}
        \item The cocharacter $\mu$ factors through $L$.
        \item Let $[b]$ be the $\sigma$-conjugacy class of $b$ in $L(\breve{\Q}_p)$. Then $[b]\cap L(\Breve{\Z}_p)\mu(p)L(\Breve{\Z}_p)$ is not empty.
        \item In the action of $\mu$ and $\nu_{b}$ on $\mathrm{Lie}(U)\otimes_{\Q_p}\Breve{\Q}_p$, only non-negative characters occur.
    \end{enumerate}
\end{definition}

We recall the following concrete example from \cite[Section 4.2.1]{HongPublished}. 
\begin{example} \label{eg: GL_n}
    Let $G = \mathrm{Res}_{ \cO_F/\Z_p}\GL_n$, where $ \cO_F$ is the ring of integers of a finite extension $F$ of $\Q_p$. For a partition $n= \sum_{k=1}^r j_k$, the Levi factor $L$ is of the form
\[L = \prod_{k=1}^{r}\mathrm{Res}_{\cO_F/\Z_p}\mathrm{GL}_{j_k}\]
Moreover, we have the following identification of the set of Newton points \cite[Example 1.10]{rapoport1996classification}:
\[\mathcal{N}(G) = \{(\lambda_1,...,\lambda_n) \in \Q^n \;:\; \lambda_1 \leq \lambda_2\leq...\leq \lambda_n\}\]
Using the Newton map, we can obtain the slopes of the Newton polygon $(\nu_1,...,\nu_n)$ from $[b]$, and the slopes of the $\sigma$-invariant Hodge polygon $\overline{\mu}=(\mu_1,...,\mu_r)$ from $\{\mu\}$. Let 
\[m_c := \sum_{k=1}^c j_k \;\text{ for }0 \leq c \leq r\]

Let $\nu_{m_c}$ denote the \emph{last slope} of the $c^{\mathrm{th}}$-block, and let $\nu_{m_c+1}$ denote the \emph{first slope} of the $(c+1)^{\mathrm{th}}$-block. This distinction is important as a given block might have multiple slopes. The local Shimura datum $(\mathrm{Res}_{ \cO_F/\Z_p}\GL_n, [b], \{\mu\})$ is Hodge--Newton reducible with respect to $L$ when:
\begin{itemize}
   \item The Newton polygon meets the Hodge polygon in certain points specified by the Levi subgroup:
   \begin{align*}
       \sum_{a=1}^{m_c} \nu_a = \sum_{b=1}^c \mu_b \; \text{ for each } 1 \leq c \leq r.
   \end{align*}
    \item These points are break points of the Newton polygon:
    \begin{align*}
        \nu_{m_c} < \nu_{m_c + 1} \; \text{ for each }0 \leq c \leq r-1
    \end{align*}
\end{itemize}
\end{example}

\begin{remark}
    A purely group-theoretic definition of Hodge--Newton reducibility is given in \cite[Definition 3.1]{chen2023fargues} and \cite[Definition 2.1]{chen2022weakly}, which is applicable in a more general setting involving local shtuka spaces. As proved in \cite[Lemma 3.3]{chen2023fargues}, it reduces to the definition mentioned above when $G$ is quasi-split.
\end{remark}

\subsection{The conjecture and the strategy}
We now state the Harris--Viehmann conjecture for Hodge--Newton reducible Rapoport--Zink spaces of Hodge type.

The method used in \cite{HongPublished} (originally devised in \cite{mantovan2008non}) is as follows:
\begin{enumerate}
     \item By the existence of the EL-realization, we have $\widetilde{G}$ and corresponding Levi subgroup $\widetilde{L}$ coming from a fixed parabolic $\widetilde{P}$. The closed embedding $G \hookrightarrow \widetilde{G}$ induces the closed embedding of Rapoport--Zink spaces of Hodge type $\RZ_{G,b} \hookrightarrow \RZ_{\widetilde{G}, b}$ by \cite{kim2018rapoport}. In \cite{mantovan2008non}, $\widetilde{P}$ is used to construct the formal scheme $\RZ_{\widetilde{P}, b}$. By \cite{HongPublished}, the pullback of $\RZ_{\widetilde{P}, b}$ over $\RZ_{G,b}$ gives rise to the space $\RZ_{P,b}$, which is the analogue of a Rapoport--Zink space corresponding to a fixed parabolic $P$ of $G$.
     
    \item It is then shown that the rigid analytic generic fiber of the space $\RZ_{G,b}$ is parabolically induced from the rigid analytic generic fiber of $\RZ_{L,b}$. This is proved by:
    \begin{lemma} \cite[Lemma 4.3.1]{HongPublished} Consider the maps $\pi_1: \RZ_{P,b} \to \RZ_{L,b}$ and $\pi_2: \RZ_{P,b} \to \RZ_{G,b}$. We abuse notation to refer to the corresponding maps on the respective rigid analytic generic fibres. Then, the rigid generic fibres of $\RZ_{G,b}$, $\RZ_{P,b}$ and $\RZ_{L,b}$ fit into the following diagram
        \[
        \begin{tikzcd}
        & {\RZ_{P,b}^{\mathrm{rig}}} \arrow[ld, "\pi_1"] \arrow[rd, "\pi_2"] & \\
        {\RZ_{L,b}^{\mathrm{rig}}} \arrow[ru, "s", bend left] & & {\RZ_{G,b}^{\mathrm{rig}}}
        \end{tikzcd}
        \]
        such that:
        \begin{enumerate}
            \item $s$ is a closed immersion.
            \item $\pi_1$ is a fibration in rigid-analytic balls.
            \item $\pi_2$ is an isomorphism.
        \end{enumerate}
    \end{lemma}
    
    \item The (generalized) cohomologies of $\RZ_{G,b}$, $\RZ_{L,b}$ with $\RZ_{P,b}$ are compared via this lemma.
   
\end{enumerate}

\begin{conjecture}(Harris--Viehmann conjecture) \cite[Conjecture 8.5]{RapoportViehmann}\label{cnj: HV OG}
    Fix a prime $\ell \neq p$. Let $(G, [b], \{\mu\})$ be a non-basic unramified local Shimura datum. Let $L$ be a Levi subgroup of $G$ such that $b \in L(\breve{\Q}_p)$ for some $b \in [b]$, and let $P$ be a standard parabolic containing $L$. Let $J_b$ be an inner form of a Levi subgroup contained in $L$. Let $I_{b,\{\mu\}, L}$ be the set of $L(\breve{\Z}_p)$-conjugacy classes of $L$ with cocharacter representative $\mu'$ such that:
\begin{enumerate}
    \item $\mu' \in \{\mu\}$.
    \item $[b] \cap L(\breve{\Z}_p)\mu'(p)L(\breve{\Z}_p)$ is non-empty.
\end{enumerate}
 Then for any admissible $\overline{\Q}_\ell$-representation $\rho$ of $J_b(\Q_p)$, the virtual representation  of $G(\Q_p)\x \mathcal{W}_E$ is expressed by the following parabolic inductive formula:
    \begin{align*}
        \mathrm{H}^\bullet(\RZ^\infty_{G, b})_\rho = \mathrm{Ind}_{P(\Q_p)}^{G(\Q_p)}\Big(\sum_{\{\mu'\}\in I_{b,\{\mu\}, L}}\mathrm{H}^\bullet(\RZ^\infty_{L, b})_\rho \Big).
    \end{align*}
\end{conjecture}

By \cite[Theorem 8.10]{RapoportViehmann}, it follows that  $I_{b,\{\mu\}, L}$ is a singleton when the local Shimura datum $(G, [b], \{\mu\})$ is unramified and of Hodge type. Thus, the precise statement of the Harris--Viehmann conjecture that is proved is in \cite[Theorem 1]{HongPublished} as follows:

\begin{theorem}
    Fix a prime $\ell \neq p$. Let $(G, [b], \{\mu\})$ be a non-basic unramified local Shimura datum of Hodge type which is Hodge--Newton reducible with respect to a parabolic subgroup $P$ of $G$ with Levi factor $L$.  For any admissible $\overline{\Q}_\ell$-representation $\rho$ of $J_b(\Q_p)$, the virtual representation  of $G(\Q_p)\x \mathcal{W}_E$ is expressed by the following parabolic inductive formula:
    \begin{align*}
        \mathrm{H}^\bullet(\RZ^\infty_{G, b})_\rho = \mathrm{Ind}_{P(\Q_p)}^{G(\Q_p)}\mathrm{H}^\bullet(\RZ^\infty_{L, b})_\rho
    \end{align*}
    It follows that the virtual representation $\mathrm{H}^\bullet(\RZ^\infty_{G, b})_\rho$ contains no supercuspidal representations of $G(\Q_p)$.
\end{theorem}

\section{Rapoport--Zink spaces of abelian type}\label{Sec: 5}

In this section, we review Shen's construction of Rapoport--Zink spaces of abelian type in \cite{ShenAbelianRZ}.

Consider an \textit{unramified} local Shimura datum  $(G, [b_G], \{\mu_G\})$; we will also write $G$ for the hyperspecial integral model of $G$ over $\Z_p$. $G(\breve{\Q}_p)$ admits the Cartan decomposition 
\begin{equation*}
    G(\breve{\Q}_p)=\coprod_{\{\mu\}}G(W)\mu(p)G(W).
\end{equation*}
where the union is over all conjugacy classes of cocharacters of $G$.
The \textit{Kottwitz map} (c.f. \cite{kottwitz1997isocrystals})
\begin{align*}
   \kappa_G: G(\breve{\Q}_p) \longrightarrow \pi_1(G)
\end{align*}
sends $g \in G(W)\mu(p)G(W)$ to the class of $\mu$. There exists an element $c_{b_G, \mu_G} \in \pi_1(G)$ such that $\kappa_G(b_G)-\mu_G = (1-\sigma)c_{b_G, \mu_G}$, with a unique $\pi_1(G)^{\Gamma}$-coset.

Fix a representative $b_G \in G(\breve{\Q}_p)$ of the class $[b_G]\in B(G)$. We get the \emph{affine Deligne-Lusztig variety (ADLV)} of the local Shimura datum:
\begin{align*}
    X_{\mu_G}^{G}(b_G)=\{g \in G(\breve{\Q}_p)/G(W) \;:\; g^{-1}b_G\sigma(g) \in G(W)\mu_G(p)G(W)\}
\end{align*}
We note that $X_{\mu_G}^{G}(b_G)$ admits a natural action by $J_{b_G}(\Q_p)$. Similarly, from the unramified local Shimura datum $(H, [b_H], \{\mu_H\})$ of Hodge type  associated to $(G, [b_G], \{\mu_G\})$, we get the ADLV $X_{\mu_H}^{H}(b_H)$. The ADLVs have natural structures of perfect schemes \cite{zhu2017affine}. By \cite{kim2018rapoport}, $X_{\mu_H}^{H}(b_H)$ is isomorphic to $\overline{\RZ}_H^{\mathrm{perf}}$, the perfection of the special fiber of $\RZ_H$.

The Kottwitz map $\kappa_H:H(\breve{\Q}_p)\to \pi_1(H)$ induces a map $\omega_H: X_{\mu_H}^{H}(b_H) \rightarrow c_{b_H, \mu_H}\pi_1(H)^{\Gamma}$ of \'etale sheaves.
Let $X_{\mu_H}^{H}(b_H)^+ \subset X_{\mu_H}^{H}(b_H)$ be its fiber over $c_{b_H, \mu_H}$, and $\RZ_H^+ \subset \RZ_H$ be the corresponding formal subscheme. Similarly, we have $X_{\mu_G}^{G}(b_G)^+ \subset X_{\mu_G}^{G}(b_G)$ as the fibre of the map $\omega_G: X_{\mu_G}^{G}(b_G) \rightarrow c_{b_G, \mu_G}\pi_1(G)^{\Gamma}$ over $c_{b_G, \mu_G}$. 

By \cite[Corollary 2.6]{ShenAbelianRZ}, $X_{\mu_G}^{G}(b_G)^+ \cong X_{\mu_H}^{H}(b_H)^+$. This motivates setting $\RZ_G^+:=\RZ_H^+$. Moreover, by \cite[Theorem 2.1]{ShenAbelianRZ}, $X_{\mu_G}^{G}(b_G) = J_{b_G}(\Q_p)X_{\mu_G}^{G}(b_G)^+$. We define 
\begin{align*}
    \RZ_G:= J_{b_G}(\Q_p)\RZ_G^+ \cong \bigsqcup_{J_{b_G}(\Q_p)/J_{b_G}(\Q_p)^+}\RZ_G^+
\end{align*}
where $J_{b_G}(\Q_p)^+$ is the stabilizer of $X_{\mu_G}^{G}(b_G)^+$ in $J_{b_G}(\Q_p)$. By construction, $\overline{\RZ}_G^{\mathrm{perf}}$ is isomorphic to $  X_{\mu_G}^{G}(b_G)$.

This construction is independent of the auxiliary local Shimura data of Hodge type, as there is an isomorphism $\RZ_{H_1}^+ \cong \RZ_{H_2}^+$ for two distinct such choices $(H_1, [b_{H_1}], \{\mu_{H_1}\})$ and $(H_2, [b_{H_2}], \{\mu_{H_2}\})$. 

Assuming that the datum $(G, [b_G], \{\mu_G\})$ admits an associated Hodge-type datum $(H, [b_H], \{\mu_H\})$ such that the map $\pi_1(H)^\Gamma\to \pi_1(H^{\ad})^\Gamma$ is surjective, we have the following identifications by \cite[Proposition~4.9]{ShenAbelianRZ}:
\begin{enumerate}
    \item $\RZ_{G^{\ad}} \cong \RZ_H/X_*(Z_H)^{\Gamma}$.
    \item $\RZ_G \cong \beta_G^*(\RZ_{G^{\ad}})$, where $\beta_G: \pi_1(G)^{\Gamma} \to \pi_1(G^{\ad})^{\Gamma}$ is the map induced by the quotient map $G\twoheadrightarrow G^{\ad}$.
\end{enumerate}

\begin{lemma}
    Let $(G, [b_G], \{\mu_G\})$ be an unramified local Shimura datum of abelian type, and $(H, b_H, \{\mu_H\})$ be the associated local Hodge type Shimura datum. Assume the natural map $\pi_1(H)^{\Gamma} \twoheadrightarrow \pi_1(H^{\ad})^{\Gamma}$ is surjective. Then we have the following isomorphisms between formal schemes over $\Spf\breve{\Z}_p$:
    \begin{align*}
        \RZ_G^+ = \RZ_H^+\cong \RZ_H/\pi_1(H)^{\Gamma} \cong \RZ_{H^{\ad}}/\pi_1(H^{\ad})^{\Gamma}.
    \end{align*}
\end{lemma}
\begin{proof}
    By construction, $\RZ_G^+ = \RZ_H^+$. The right-action of $\pi_1(H)^{\Gamma}$ on $c_{b_H, \mu_H}\pi_1(H)^{\Gamma}$ induces an action on $\RZ_H$. More precisely, any $g \in \pi_1(H)^{\Gamma}$ induces an isomorphism 
    \begin{equation*}
        R_g:c_{b_H, \mu_H}\pi_1(H)^{\Gamma}\xrightarrow{\sim} c_{b_H, \mu_H}\pi_1(H)^{\Gamma}.
    \end{equation*}
    This induces a Cartesian diagram 
    
    \[\begin{tikzcd}
	{\RZ_H} & {\RZ_H} \\
	{c_{b_H, \mu_H}\pi_1(H)^{\Gamma}} & {c_{b_H, \mu_H}\pi_1(H)^{\Gamma}}
	\arrow["{R_g'}", from=1-1, to=1-2]
	\arrow["{\omega_H}", from=1-1, to=2-1]
	\arrow["{\omega_H}", from=1-2, to=2-2]
	\arrow["{R_g}", from=2-1, to=2-2]\\
    \end{tikzcd}\]
    The isomorphism $R_g'$ is the induced action of $g$ on $\RZ_H$.
    
    In particular, all fibres of $\omega_H$ are identified under the action of $\pi_1(H)^{\Gamma}$. Thus, $\RZ^+_G = \RZ^+_H \cong \RZ_H/\pi_1(H)^{\Gamma}$.
    By \cite[Proposition 4.9]{ShenAbelianRZ}, we have 
    \begin{equation*}
        \RZ_H/X_*(Z_H)^{\Gamma} \cong \RZ_{G^{\ad}}.
    \end{equation*}
    By \cite[Lemma 1.5]{Borovoi1989preprint}, we have a short exact sequence:
    \[0 \longrightarrow X_*(Z_H) \longrightarrow \pi_1(H) \longrightarrow \pi_1(H^{\ad}) \longrightarrow 0\]
    We apply the left exact functor $M \mapsto M^{\Gamma}$ in conjunction with our assumption $\pi_1(H)^{\Gamma} \twoheadrightarrow \pi_1(H^{\ad})^{\Gamma}$ to get
    \[0 \longrightarrow X_*(Z_H)^{\Gamma} \longrightarrow \pi_1(H)^{\Gamma} \longrightarrow \pi_1(H^{\ad})^{\Gamma} \longrightarrow 0\]
    In particular, we see that
    \begin{align*}
    \RZ_G^+ = \RZ_H^+ \cong \RZ_H/\pi_1(H)^{\Gamma} \cong (\RZ_H/X_*(Z_H)^{\Gamma})/\pi_1(H^{\ad})^{\Gamma} \cong \RZ_{G^{\ad}}/\pi_1(H^{\ad})^{\Gamma}.
    \end{align*}
\end{proof}

\begin{remark} \label{rmk: LES vanishing}
    The condition that the map $\pi_1(H)^\Gamma \to \pi_1(H^{\ad})^\Gamma$ is a surjection is true for many unramified reductive groups $H$. For example, it is true if $H$ is a similitude group. This includes $\GL_n,\mathrm{GSp}_{2n}, \mathrm{GSpin}_{n}$, etc. However, the map $\pi_1(H)^\Gamma \to \pi_1(H^{\ad})^\Gamma$ fails to be surjective if $G=\SL_n$.
\end{remark}

\section{Weil descent datum}\label{Sec: 6}

In this subsection, we discuss how the Weil descent data work for Rapoport--Zink spaces of abelian type. In particular, $\mathrm{H}^\bullet(\RZ_G^\infty)$ will be a virtual representation of $G(\Q_p) \times \mathcal{W}_E$, where $\mathcal{W}_E$ is the Weil group associated to the reflex field $E = E(\mu_G)$ of the local Shimura datum of abelian type $(G, [b_G], \{\mu_G\})$. We first recall the formalism of Weil descent datum in terms of the moduli space of $p$-adic local shtukas as in \cite{PappasRapoport2024integral}, which we include for the sake of completeness. Following that, we recount Kim's construction for the specific case of local Shimura data of Hodge type, as in \cite{kim2018rapoport}. We then define the Weil descent datum for local Shimura data of abelian type, and show that it agrees with the Pappas-Rapoport construction.

\subsection{Weil descent datum by $p$-adic local shtukas}\label{sec:RP-Weil}

In the following, let $\cG$ be a parahoric group scheme over $\Z_p$ with generic fiber $G$. In many cases, we shall assume $\cG$ is hyperspecial (and $G$ is thus unramified).
Let $\cM^{\mathrm{int}}_{(\cG, [b_G], \{\mu_G\})}$ be the integral moduli space of shtukas associated to an integral local Shimura datum $(G,[b_G],\{\mu_G\})$ and $\cG$ as defined in \cite[Lecture 25.1]{BerkeleyLectures}. It is the functor that sends $S \in \mathrm{Perfd}_k$ (here $k$ is the residue field of $\breve{E}$) to the set of isomorphism classes of tuples $(S^\#, \mathscr{P}, \phi_{\mathscr{P}}, i_r)$ where:
\begin{enumerate}
    \item $S^\#$ is an untilt of $S$ over $\mathrm{Spa}(O_{\breve{E}})$,
    \item $(\mathscr{P}, \phi_{\mathscr{P}})$ is a $\mathcal{G}$-shtuka over $S$ with one leg along $S^\#$ bounded by $\mu$,
    \item $i_r: G_{\mathcal{Y}_{[r,\infty)}(S)}\xrightarrow{\cong}\mathscr{P}|_{\mathcal{Y}_{[r,\infty)}(S)} $ is a framing isomorphism of $G$-torsors for large enough $r$, under which $\phi_{\mathscr{P}}$ is identified with $\phi_b = b \times \mathrm{Frob}_S$.
\end{enumerate}

In \cite[\S 3.1,\S 3.2]{PappasRapoport2024integral}, Pappas and Rapoport define a natural Weil datum on $\cM^{\mathrm{int}}_{(\cG, [b_G], \{\mu_G\})}$ from $\cO_{\breve{E}}$ to $\cO_{E}$, which we recall now.
Let $\tau$ be the relative Frobenius automorphism of $\breve{E}$ over $E$, and define the $v$-sheaf $\cM^{\mathrm{int},(\tau)}_{(\cG, [b_G], \{\mu_G\})}$ by 
\begin{equation*}
    \cM^{\mathrm{int},(\tau)}_{(\cG, [b_G], \{\mu_G\})}(S):=\cM^{\mathrm{int},(\tau)}_{(\cG, [b_G], \{\mu_G\})}(S\times_{\Spa(k),\tau}\Spa(k)).
\end{equation*}

Let $S = \mathrm{Spa}(R, R^+)$, with $k$-algebra structure $\varepsilon: k \to R$. Let $R_{[\tau]}$ be the same ring with $k$-algebra structure $ k \xrightarrow{x \mapsto x^q} k \xrightarrow{\varepsilon} R$ for $q = |\kappa_E| = p^d$. In particular, $\mathrm{Spa}(R_{[\tau]}, R^+_{[\tau]}) = S \times_{\mathrm{Spa}(k),\tau} \mathrm{Spa}(k)$. Then the Weil datum $\Phi^{\mathrm{PR}}$ is:
\begin{align*}
    w: \cM^{\mathrm{int}}_{(\cG, [b_G], \{\mu_G\})} \longrightarrow \cM^{\mathrm{int}, (\tau)}_{(\cG, [b_G], \{\mu_G\})} 
\end{align*}
which sends a point $(S^\#, \mathscr{P}, \phi_{\mathscr{P}}, i_r)$ valued in $S$ to a point  $(S^\#, \mathscr{P}, \phi_{\mathscr{P}}, i'_{qr})$ valued in $\mathrm{Spa}(R_{[\tau]}, R^+_{[\tau]})$. The new framing $i'_{qr}$ is defined as the composition
\begin{align*}
    G_{\mathcal{Y}_{[qr, \infty)}(R_{[\tau]}, R^+_{[\tau]})} \xrightarrow{(\phi_b)^{d}} G_{\mathcal{Y}_{[r,\infty)}(R, R^+)} \xrightarrow{\cong}\mathscr{P}|_{\mathcal{Y}_{[r,\infty)}(R, R^+)}\xrightarrow[\cong]{\phi}\mathscr{P}|_{\mathcal{Y}_{[qr,\infty)}(R_{[\tau]}, R^+_{[\tau]})}.\\
\end{align*}

\subsection{Weil descent data via Shen's construction}

There is yet another natural way of constructing Weil descent data on $\RZ_G$ following Shen's construction of $\RZ_G$. The idea is as follows: A Weil descent datum in the case $G=\GL_n$ is discussed in \cite{rapoport1996period}. We can define a Weil descent datum on a Rapoport--Zink space $\RZ_H$ of Hodge type by restriction along a Hodge embedding $H\hookrightarrow \GL_n$. Assuming $\RZ_H$ is associated to $\RZ_G$, we can then "transfer" this Weil descent datum from $\RZ_H$ to $\RZ_G$. We note that the above construction recovers the case of local Shimura varieties and Rapoport--Zink spaces of Hodge type from the local Shimura datum $(H, [b_H], \{\mu_H\})$.

We first review the Weil descent data in the Hodge type case as constructed in \cite[\S 7.3]{kim2018rapoport}. Assume $(H,[b_H],\{\mu_H\})$ is an unramified local Shimura datum of Hodge type, and fix a \textit{Hodge embedding} $\iota: (H,[b_H],\{\mu_H\}) \to (\GL_n,[b'],\{\mu'\})$, i.e., $\iota$ is a map of local Shimura data induced by a closed embedding $\iota:H \to \GL_n$ of groups. Then by the construction of $\RZ_H$ in \cite{kim2018rapoport}, we get a closed immersion of formal groups $\iota_{\RZ}:\RZ_H\to\RZ_{\GL_n}$. The latter space $\RZ_{\GL_n}$ has a natural and explicit Weil descent datum in terms of $p$-divisible groups. Kim in \cite{kim2018rapoport} then defines the Weil descent datum on $\RZ_H$ to be the restriction of that on $\RZ_{\GL_n}$ along the closed immersion $\iota_{\RZ}$. We denote this Weil descent datum as 
\begin{equation*}
    \Phi^{\mathrm{Kim}}_{(H,\iota)}:\RZ_H\xrightarrow{\sim} \RZ_H^{(\tau)},
\end{equation*}
where, as before, $\tau$ is the relative Frobenius on $\breve{E}$ over $E$, and $\RZ_H^{(\tau)}:=\RZ_H\times_{\Spf\cO_{\breve{E}},\tau}\Spf\cO_{\breve{E}}$. As the notation suggests, $\Phi^{\mathrm{Kim}}_{(H,\iota)}$ a priori depends on the choice of Hodge embedding $\iota$. We check in the next lemma that it is independent of this choice.

\begin{lemma}\label{lem:independence-Hodge-Weil}
    The Weil descent datum $\Phi^{\mathrm{Kim}}_{(H,\iota)}$ is independent of the Hodge embedding $\iota:H\hookrightarrow \GL_n$.
\end{lemma}
\begin{proof}
    Let $(H,[b_H],\{\mu_H\})$ be an unramified local Shimura datum of Hodge type. Let $\iota: H \hookrightarrow \GL_n$ and $\iota':H\hookrightarrow \GL_m$ be two Hodge embeddings. We get a third Hodge embedding $\iota'':H\hookrightarrow\GL_{n+m}$ by taking the product $\iota'':=\iota\times\iota'$. 
    It then suffices to verify that the Weil descent data defined by $\iota$ and $\iota''$ coincide, which is clear since the induced embedding $f:\GL_n\to \GL_{n+m}$ commutes with Weil descent data on both sides.
\end{proof}

\begin{theorem}\label{thm:wdd-PR-Kim}
    The two Weil descent data $\Phi^{\mathrm{PR}}$ and $\Phi^{\mathrm{Kim}}$ coincide for (unramified) local Shimura data of Hodge type.
\end{theorem}
\begin{proof}

By \Cref{lem:independence-Hodge-Weil}, it is enough to show this for the case $G=\GL_n$. 

The effect of the Weil descent datum on Kim's construction is as follows (using the convention of \cite{BerkeleyLectures}): $(X, \xi)\mapsto (X^{(\tau)}, \xi^{(\tau)}) = (X, \xi \circ \phi_b^{d})$, where $\xi: X \to  \mathbb{X}$ is the quasi-isogeny providing the framing structure, and $\xi^{(\tau)}: X \to \tau^*X \to \tau^*\mathbb{X} $. The point $(X, \xi)$ corresponds to a point in the moduli space of shtukas $\mathrm{Sht}_{(\mathcal{G}, [b_G], \{\mu_G\})}$ given by $(S^{\#}, \mathcal{P}, \varphi_{\mathcal{P}}, i_r)$ in Pappas-Rapoport notation. The point $(X^{(\tau)}, \xi^{(\tau)})$ corresponds to a point in $\mathrm{Sht}_{(\mathcal{G}, [b_G], \{\mu_G\})}$ with the framing modified by Frobenius to the $d$th power, given by $(S^{\#}, \mathcal{P}, \varphi_{\mathcal{P}}, i_r\circ \phi_b^{d})$. This agrees with the map in Pappas-Rapoport:
\begin{align*}
    \mathrm{Sht}_{(\mathcal{G}, [b_G], \{\mu_G\})} \longrightarrow \mathrm{Sht}_{(\mathcal{G}, [b_G], \{\mu_G\})}^{(\tau)} \\
    (S^{\#}, \mathcal{P}, \varphi_{\mathcal{P}}, i_r) \mapsto (S^{\#}, \mathcal{P}, \varphi_{\mathcal{P}}, i'_{r'})
\end{align*}
\end{proof}

\begin{theorem}\label{thm: wdd-Shen}
    Let $(G, [b_G], \{\mu_G\})$ be an unramified local Shimura datum of abelian type, $\RZ_G$ be its Rapoport--Zink space, and $E_G$ be its reflex field. 
    There exists a Weil descent datum $\Phi^{\mathrm{Shen}}$ for $\RZ_{G}$ relative to some finite extension $E^{\mathrm{opt}}$ of $E_G$, which depends only on $(G, [b_G], \{\mu_G\})$. When $\RZ_{(G, [b_G], \{\mu_G\})}$ is of Hodge type, then $E^{\mathrm{opt}}=E_G$, and we recover $\Phi^{\mathrm{Shen}} = \Phi^{\mathrm{Kim}}$.
\end{theorem}

\begin{proof}

For a local Shimura datum $(H, [b_H], \{\mu_H\})$, write $E_H$ for its reflex field. Let $E^{\mathrm{Hdg}}$ be the intersection of all $E_H$ for all $(H, [b_H], \{\mu_H\})$ unramified of Hodge type associated to $(G,[b_G],\{\mu_G\})$, for each of which the map $\beta_H: \pi_1(H)^{\Gamma} \to \pi_1(H^{\ad})^\Gamma$ is a surjection. Let $E^{\mathrm{opt}}$ be the compositum of $E^{\mathrm{Hdg}}$ and $E_G$. Fix a datum $(H, [b_H], \{\mu_H\})$ from the family above such that $E_H=E^{\mathrm{Hdg}}$. From the compatibility between $\Phi^{\mathrm{PR}}$ and $\Phi^{\mathrm{Kim}}$, the action of the Weil descent datum $\tau_H$ must be equivariant with the morphism of \'etale sheaves $\omega_H: \RZ_H \to \pi_1(H)^{\Gamma}$. 

Write $\Phi_H:\RZ_H \xrightarrow{\sim} \RZ_H^{(\tau_H)}$ for the Weil descent datum relative to $E_H$ as in \Cref{thm:wdd-PR-Kim}. Since $\RZ_{H^\ad}\cong \RZ_H/X_*(Z_H)^\Gamma$, taking quotients by $X_*(Z_H)^\Gamma$ on both sides gives a Weil descent datum 
\begin{equation*}
    \Phi^{\ad}:\RZ_{H^{\ad}} \xrightarrow{\sim} \RZ_{H^{\ad}}^{(\tau_H)}
\end{equation*}
relative to $E_H$.
Note that since $E_{H^{\ad}}$ is contained in $E_{H}$ and both fields are unramified over $\Q_p$, the Frobenius $\tau_H$ is a multiple of $\tau_{H^{\ad}}$. Since $E^{\mathrm{opt}}$ is a finite unramified extension of $E_H$, self-composition of $\Phi^{\ad}$ yields a Weil descent datum 
\begin{equation*}
    \Phi^{\mathrm{opt}}:\RZ_{H^{\ad}} \xrightarrow{\sim} \RZ_{H^{\ad}}^{(\tau_{\mathrm{opt}})},
\end{equation*}
where $\tau_{\mathrm{opt}}$ is the Frobenius on $E^{\mathrm{opt}}$.
Let $\beta_G: \pi_1(G)^{\Gamma} \to \pi_1(G^{\ad})^\Gamma$ be the natural map. Since $\RZ_G \cong \beta_G^{*}(\RZ_{G^{\ad}})$ and $\RZ_{H^{\ad}}\simeq \RZ_{G^{\ad}}$, pulling back $\Phi^{\mathrm{opt}}$ along $\beta_G$ yields a Weil descent datum
\begin{align*}
    \Phi^{\mathrm{Shen}}:\RZ_G\xrightarrow{\sim}\RZ_G^{(\tau_{\mathrm{opt}})}.
\end{align*}
\end{proof}

In particular, this construction is equivariant with respect to the morphism of \'etale sheaves $\omega_G: \RZ_G \to \pi_1(G)^{\Gamma}$. Hence, $\varinjlim_{K} \mathrm{H}^i(\RZ_G^K)$ is endowed with a natural action of $\mathcal{W}_{E^{\mathrm{opt}}}$, the Weil group of $E^{\mathrm{opt}}$.

\begin{lemma}\label{lem:independence-Shen-Weil}
    The Weil descent datum $\Phi^{\mathrm{Shen}}_{(G,\iota)}$ is independent of the auxiliary choice of Hodge type datum $(H, [b_H], \{\mu_H\})$.
\end{lemma}
\begin{proof}

Let $(H_1, [b_{H_1}], \{\mu_{H_1}\})$ and $(H_2, [b_{H_2}], \{\mu_{H_2}\})$ be two distinct auxiliary local Shimura data of Hodge type associated to the same local Shimura datum $(G, [b_G], \{\mu_G\})$ of abelian type. By the proof of \cite[Theorem 4.6]{ShenAbelianRZ}, we have a canonical identification $\RZ_{H_1}^+ \cong \RZ_{H_2}^+ =: \RZ_G^+$. The construction of the Weil descent datum in \Cref{thm: wdd-Shen} follows.
\end{proof}

Let $\Phi^{\mathrm{PR}}$ be the Weil descent datum on $\RZ_G$ relative to $E_G$ defined in \Cref{sec:RP-Weil}. Thus, $\Phi^{\mathrm{PR}}$ defines a Weil descent datum $\Phi^{\mathrm{PR}}\otimes_{E_G}E^{\mathrm{opt}}$ relative to the finite extension $E^{\mathrm{opt}}$.
\begin{corollary}
    The two Weil descent data $\Phi^{\mathrm{PR}}\otimes_{E_G}E^{\mathrm{opt}}$ and $\Phi^{\mathrm{Shen}}$ relative to $E^{\mathrm{opt}}$ coincide for (unramified) local Shimura data of abelian type. 
\end{corollary}
\begin{proof}

Since $\RZ_G = J_{b_G}(\Q_p) \RZ_H^+$ for some auxiliary choice of Hodge type datum $(H, [b_H], \{\mu_H\})$, it suffices to show the statement for the Hodge case. This follows from \Cref{thm:wdd-PR-Kim}. \end{proof}

\section{Relations to the Hodge case for Hodge--Newton reducibility}\label{Sec: 7}

Recall our strategy is to use the known results of Harris--Viehmann in the Hodge--Newton reducible Hodge type case to deduce results for the Hodge--Newton reducible abelian type case. The following lemma shows that we can always find an associated local Shimura datum of Hodge type that is Hodge--Newton reducible. 

\begin{lemma} \label{lem: HN ab Hodge}
    If a local abelian type Shimura datum $(G, [b_G], \{ \mu_G\})$ is Hodge--Newton reducible, then there exists an associated local Shimura datum $(H, [b_H], \{ \mu_H\})$ of Hodge type that is also Hodge--Newton reducible.
\end{lemma}

\begin{proof}
    Since $(G, [b_G], \{ \mu_G\})$ is Hodge--Newton reducible (with respect to some fixed parabolic $P_G$ and Levi $L_G$), by definition, there exists an induced datum $(L_G, [b_{G}], \{ \mu_{G}\})$ for the Levi $L_G$. 
     Let $L_{G^{\ad}}$ (resp. $P_{G^{\ad}}$) be the image of $L_G$ (resp. $P_G$) along the quotient map $q_G: G\twoheadrightarrow G^{\ad}$. We now check the local Shimura datum $(G^{\ad}, [b_{G^{\ad}}], \{\mu_{G^{\ad}}\}])$ induced by taking adjoint quotient is again Hodge--Newton reducible. In fact, Properties (1) and (2) in Definition~\ref{def:HN-reducible} are evident. Property (3) holds since the unipotent radical $U_G$ is isomorphic to its image under the map $q_G: G\twoheadrightarrow G^{\ad}$. Thus, we obtain datum $(L_{G^{\ad}}, [b_{G^{\ad}}], \{ \mu_{G^{\ad}}\})$ for the Levi $L_{G^{\ad}}$ of $G^{\ad}$.
     
    By definition, there is an isomorphism $\Xi:(G^{\ad}, [b_{G^{\ad}}], \{\mu_{G^{\ad}}\}]) \cong (H^{\ad}, [b_{G^{\ad}}], \{ \mu_{H^{\ad}}\})$ of local Shimura data. We also write $\Xi:G^{\ad}\xrightarrow{\sim} H^{\ad}$ for the underlying group isomorphism. Now let $L_H$ be the preimage of $L_{H^{\ad}}:=\Xi(L_{G^{\ad}})$ under $q_H: H\twoheadrightarrow H^{\ad}$; so $L_H$ is a Levi of $H$. We want to show that some $b_{H}$ and minuscule $\mu_{H}$ such that $(H, [b_H], \{\mu_H\})$ is Hodge--Newton reducible with respect to $(L_H, [b_{H}], \{ \mu_{H}\})$. 
    
    We first focus on the cocharacter $\mu_{H}$. Let $g \in \GG_m$. By the Hodge--Newton reducibility of the adjoint datum, we saw that $\mu_{G^{\ad}} \cong \mu_{H^{\ad}}$ factors through the Levi $L_{G^{\ad}}$ of $G^{\ad}$. Then by definition, $\mu_H(g) \in q_H^{-1}(\mu_{H^{\ad}}) \subset L_H$. Thus, $\mu_{H}: \GG_m \to H$ factors through $\GG_m \xrightarrow{\mu_{L_H}}L_H \hookrightarrow H$. For any $\lambda \in X_*(Z(L_H))$, the pairing with any weight coming from $\mu_{H}$ is trivial. Thus, $\mu_{H}$ is minuscule since $\{\mu_{G^{\ad}}\}$ is minuscule.

    To show that $b_H$ satisfies \Cref{def:HN-reducible}, let 
    \begin{equation*}
        x \in [b_{G^{\ad}}] \cap L_{G^{\ad}}(\breve{\Z}_p) \mu_{G^{\ad}}(p) L_{G^{\ad}}(\breve{\Z}_p) = [b_{G^{\ad}}] \cap L_{H^{\ad}}(\breve{\Z}_p) \mu_{H^{\ad}}(p) L_{H^{\ad}}(\breve{\Z}_p).
    \end{equation*}
    Let $y$ be a lift of $x$ in $L_H(\breve{\Z}_p) \mu_{H}(p) L_H(\breve{\Z}_p)$. Then $[b_{H}]\in B(H)$ in the datum $(H, [b_H], \{ \mu_H\})$ is the $\sigma$-conjugacy class of $y$ in $H$, thus satisfying Property (2) in Definition~\ref{def:HN-reducible} with respect to $L_H$.  \end{proof}

Recall that the set $I_{b,\{\mu\}, L}$ played an important role in the original statement of \Cref{cnj: HV OG}. We revisit this for unramified Hodge--Newton reducible local Shimura data of abelian type:

\begin{lemma}\label{lem: abelian I set}
    Let $(G, [b_G], \{\mu_G\})$ be an unramified local Shimura datum of abelian type. Suppose it is Hodge--Newton reducible with respect to parabolic $P_G$ and a Levi $L_G$. Let $(L_G, [b_G], \{\mu_{G}\})$ be the local Shimura datum corresponding to a Levi subgroup $L_G \subsetneq G$. Let $I_{b_G,\{\mu_G\}, L_G}$ be the set of $L_G(\breve{\Z}_p)$-conjugacy classes of $L_G$, as in \Cref{cnj: HV OG}. Then $I_{b_G,\{\mu_G\}, L_G}$ is a singleton, whose unique element is $\{\mu_G\}$. Moreover, the tuple $(L_G, [b_G], \{\mu_{G}\})$ is an unramified local Shimura datum of abelian type.
\end{lemma}

\begin{proof} Let $(H, [b_H], \{\mu_H\})$ be an local Shimura datum of Hodge type associated to $(G, [b_G], \{\mu_G\})$. There exists an unramified Levi subgroup $L_H \subsetneq H$, such that $(L_G, [b_G], \{\mu_{G}\})$ is associated to $(L_H, [b_H], \{\mu_{H}\})$. Associated to $I_{b_G,\{\mu_G\}, L_G}$, there is $I_{b_H,\{\mu_H\}, L_H}$. Then, by the same reasoning as in the proof of \cite[Theorem 8.10(i)]{RapoportViehmann}, we can show that $I_{b_G,\{\mu_G\}, L_G}$ is a singleton. Note that while the original set-up in \cite{RapoportViehmann} requires EL type, the proof applies to any unramified Hodge--Newton reducible local Shimura data. Consider the local abelian type datum $(L_G, [b_G], \{\mu_G\})$ and its associated local Hodge type datum $(L_H, [b_H], \{\mu_H\})$. Since $H$ is unramified, so is $L_H$. By \cite[Lemma 4.1.2]{HongPublished}, we know that $(L_H, [b_H], \{\mu_H\})$ is unramified. It follows that $(L_G, [b_G], \{\mu_G\})$ is an unramified local Shimura datum of abelian type. \end{proof}

We record the following group-theoretic lemma for later use. This should be well-known, but it seems hard to find a reference for it.
\begin{lemma}
    Let $G$ be a connected reductive group over a field $k$. Let $G^{\der}=[G,G]$ be the derived subgroup of $G$. Then there is a natural isomorphism
    \begin{equation*}
        G\cong G^{\der}\times^{Z(G^\der)} Z(G),
    \end{equation*}
    where the right-hand side is the contracted product.
    Furthermore, for any central subgroup $H$, there is a natural isomorphism
    \begin{equation*}
        Z(G/H)\cong Z(G)/H.
    \end{equation*}
\end{lemma}
\begin{proof}
    The map 
    \begin{equation*}
        \pi:G^{\der}\times^{Z(G^\der)} Z(G)\to G,\quad (g,z)\mapsto gz
    \end{equation*}
    is a well-defined homomorphism. 
    
    To see this is injective, assume $\pi(g,z)=gz=1$. So $g=z^{-1}\in Z(G^\der)$. Since the pair $(z^{-1},z)$ is equivalent to $(1,1)$ in $G^{\der}\times^{Z(G^\der)} Z(G)$, we see $\pi$ is injective. 
    
    On the other hand, $G=G^\der\cdot Z(G)$. Thus, $\pi$ is also surjective. The equality $Z(G/H)\cong Z(G)/H$ for any central $H$ is straightforward.
\end{proof}

\begin{example}
    Let $G=\GL_n$. Then $G^\der=\SL_n$, $Z(G^\der)=\mu_n$, and $Z(G)=\GG_m$.
\end{example}

\section{Functoriality of Rapoport--Zink spaces}\label{Sec: 8}

We present here some basic results concerning the functoriality of local Shimura data. Some of the results may have been known, but we include them for exposition. 

\begin{lemma}
    Let $(G, [b_G], \{\mu_G\})$ and $(G', [b_{G'}], \{\mu_{G'}\})$ be unramified local Shimura data of abelian type. Then $(G \x G', [(b_G, b_{G'})], \{(\mu,\mu')\}))$ is also an unramified local Shimura datum of abelian type. Here $[(b, b')]$ is defined to be the $\sigma$-conjugacy class of $(b,b')\in G(\breve{\Q}_p)\times G'(\breve{\Q}_p)$ for any representative $b$ of $[b]$ (resp. $b'$ of $[b']$). $(\mu_G,\mu_{G'})$ is defined similarly.

\end{lemma}

\begin{proof}
Let $(H, [b], \{\mu_H\})$ and $(H', [b'], \{\mu_{H'}\})$ be the local Shimura data of Hodge type associated with $(G, [b_G], \{\mu_G\})$ and $(G', [b_{G'}], \{\mu_{G'}\})$, respectively. By \cite[Prop 3.1.2(i)]{HongPublished}, we know that $(H\x H', [(b_H, b_{H'})], \{(\mu_H, \mu_{H'})\})$ is an unramified local Shimura datum of Hodge type. Since \[(G \x G')^{\ad} \cong G^{\ad} \x G'^{\ad} \cong H^{\ad} \x H'^{\ad} \cong (H \x H')^{\ad}, \]
we see that this is a local Hodge datum associated to the datum $(G\x G', [(b_G, b_{G'})], \{(\mu_G, \mu_{G'})\})$. Hence, the latter is unramified of abelian type. \end{proof}
    
The following proposition provides some basic functorial properties for Rapoport--Zink spaces of abelian type. This is a generalization of \cite[Theorem 4.9.1]{kim2018rapoport}.\\

\begin{proposition}\label{prop: functoriality-ab}
     Let $(G, [b_G], \{\mu_G\})$ and $(G', [b_{G'}], \{\mu_{G'}\})$ be unramified local Shimura data of abelian type. Write $\RZ_{G}$ (resp. $\RZ_{G'}$, $\RZ_{G\times G'}$) for the Rapoport--Zink space associated to the datum $(G, [b_G], \{\mu_G\})$ (resp, $(G', [b_{G'}], \{\mu_{G'}\})$, $(G \x G', [(b_G, b_{G'})], \{(\mu,\mu')\})$).
    \begin{enumerate}
    \item There is a natural isomorphism
    \begin{align*}
        \RZ_{G} \times_{\mathrm{Spf}(\breve{\Z}_p)} \RZ_{G'} \xlongrightarrow{\sim} \RZ_{G\times G',}
    \end{align*}
    
    \item For any map $f:(G,[b_G],\{\mu_G\}) \to (G',[b_{G'}],\{\mu_{G'}\})$ between local Shimura data which are unramified of abelian type, there exists an induced morphism
    \begin{align*}
        \RZ_{G} \longrightarrow \RZ_{G'}
    \end{align*}
    which is a closed embedding if the underlying group homomorphism $f:G\to G'$ is a closed embedding. 
\end{enumerate}
\end{proposition}

\begin{proof}
    Let $(H, [b_H], \{\mu_H\})$ (resp. $(H', [b_{H'}], \{\mu_{H'}\})$) be an unramified local Shimura datum of Hodge type associated to $(G,[b_G],\{\mu_G\})$ (resp. $(G',[b_{G'}],\{\mu_{G'}\})$).
    To simplify notations, we will write $\RZ_G$ for the Rapoport--Zink formal scheme associated to $(G,[b_G],\{\mu_G\})$, and similarly for other local Shimura data. We will freely use the notations from \cite{ShenAbelianRZ}
    \begin{enumerate}
        \item 
        Recall for the local Shimura datum $(H,[b_H],\{\mu_H\})$ of Hodge type, we can form a coset $c_{b_H,\mu_H}\pi_1(H)^{\Gamma}$ in the group $\pi_1(H)^{\Gamma}$; see \cite[\S 2.2]{ShenAbelianRZ} for the construction. The formation of this coset is functorial in $(H,[b_H],\{\mu_H\})$ by the functoriality of the Kottwitz map and the Galois action. Note there is an isomorphism
        \begin{equation*}
            \tau:\pi_1(H)^{\Gamma}\times\pi_1(H')^{\Gamma}\xrightarrow{\sim} \pi_1(H\times H')^{\Gamma}.
        \end{equation*}
        There is a natural isomorphism
        \begin{equation*}
            \phi: c_{b_H,\mu_H}\pi_1(H)^{\Gamma}\times  c_{b_{H'},\mu_{H'}}\pi_1(H')^{\Gamma}\xrightarrow{\sim} c_{H\times H'}\pi_{1}(H\times H')^{\Gamma}.
        \end{equation*}
        compatible with $\tau$.
        
        Fix elements $x_0\in c_{b_H,\mu_H}\pi_1(H)^{\Gamma}$ and $x_0'\in c_{b_{H'},\mu_{H'}}\pi_1(H')^{\Gamma} $, and take $y_0\coloneqq \phi(x_0,x_0')$. There is a natural map of \'etale sheaves 
        \begin{equation*}
            \omega_H:\RZ_H\to c_{b_H,\mu_H}\pi_1(H)^{\Gamma}
        \end{equation*}
        and $\RZ_H^+$ is defined to be the fiber
        \begin{equation*}
            \RZ_H^+\coloneqq \omega_H^{-1}(x_0).
        \end{equation*}
        The constructions of $\RZ_{H'}^+$ and $\RZ_{H\times H'}^+$ are similar.
        
        By the construction of the spaces $\RZ_G$ and $\RZ_{G'}$, we only need to show there is an isomorphism $$\RZ_H^+\times \RZ_{H'}^+\to \RZ_{H\times H'}^+.$$ This then follows from the natural isomorphism
        \begin{equation*}
            \RZ_H\times \RZ_{H'}\xrightarrow{\sim} \RZ_{H\times H'}
        \end{equation*}
        shown in \cite[Theorem 4.9.1]{kim2018rapoport}.

        \item 
       
            By the \cite[Proposition 1.2.1]{PappasRapoport2024}, if there exists a formal scheme that is flat, normal and locally formally of finite type whose $v$-theoretic integral model is of the form $\mathcal{M}^{\mathrm{int}}_{(\mathcal{G}, [b_G], \{\mu_G\})}$, then it must be unique. Such a formal scheme is explicitly constructed by Shen, which we have been denoting by $\RZ_{(G, [b_G], \{\mu_G\})}$.

        From $f:G \to G'$, we get an induced map of local Shimura data:$$f: (G, [b_G], \{\mu_G\}) \longrightarrow (G', [b_{G'}], \{\mu_{G'}\})$$ 
        Consider the induced map on the associated Hodge-type groups $f_H: H \to H'$ such that $b_H \mapsto f_H(b_H) = b_{H'}$. By \cite[Theorem 4.9.1]{kim2018rapoport}, this induces a map
        \begin{align*}
            \RZ_{H, b_H} \to \RZ_{H', b_{H'}},
        \end{align*}
        of the associated Rapoport--Zink spaces of Hodge type.
       
        We now use \cite[Proposition 3.6.2]{PappasRapoport2024}: a closed embedding of local Shimura data $$f: (G, [b_G], \{\mu_G\}) \longrightarrow (G', [b_{G'}], \{\mu_{G'}\})$$ induces the following closed immersion of $v$-sheaves:
        \begin{align*}
            \mathfrak{f}: \mathcal{M}^{\mathrm{int}}_{(\mathcal{G}, [b_G], \{\mu_G\})} \to \mathcal{M}^{\mathrm{int}}_{(\mathcal{G'}, [b_{G'}], \{\mu_{G'}\})}
        \end{align*}
        where $\mathcal{G}$ and $\mathcal{G'}$ are parahoric models of $G$ and $G'$ respectively.
     The statement now is a special case of Lemma \ref{lem:closed-immersion-formal-schemes} below. 
        
    \end{enumerate}

\end{proof}

\begin{lemma}\label{lem:closed-immersion-formal-schemes}
    Let $\mathcal{X}$ and $\mathcal{Y}$ be flat and normal formal schemes locally formally of finite type over $\Spf\cO_{\breve{E}}$ with smooth rigid-analytic generic fibers, and let $\mathfrak{f}:\mathcal{X}^{\Diamond}\to \mathcal{Y}^{\Diamond}$ be a morphism between the associated $v$-sheaves over $\Spd\cO_{\breve{E}}$. If $\mathfrak{f}$ is a closed immersion, then there exists a unique morphism $f:\mathcal{X}\to \mathcal{Y}$ between formal schemes over $\cO_{\breve{E}}$ such that $f$ is a closed immersion and that $f^\Diamond=\mathfrak{f}$.
\end{lemma}
\begin{proof}
    By \cite[Proposition 18.4.1]{BerkeleyLectures}, there exists a unique morphism $f:\mathcal{X}\to \mathcal{Y}$ representing $\mathfrak{f}$, i.e., $f^\Diamond=\mathfrak{f}$. We only need to verify $f$ is a closed immersion. Since the problem is Zariski-local on the target, we may assume $\mathcal{X}=\Spf A^+$ and $\mathcal{Y}=\Spf B^+$. Write $\mathcal{X}_{\eta}$ (resp. $\mathcal{Y}_{\eta}$) for the adic generic fiber of $\mathcal{X}$ (resp. $\mathcal{Y}$). 
    As pointed out in \cite[Remark 18.4.3.]{BerkeleyLectures}, $A^+$ (resp. $B^+$) can be identified as the ring of powerbounded functions on the generic fiber $\mathcal{X}_\eta$ (resp. $\mathcal{Y}_\eta$). The map $\mathfrak{f}$ induces a closed immersion $\mathfrak{f}_\eta:\mathcal{X}_\eta^\Diamond\to \mathcal{Y}_\eta^\Diamond$, and hence a surjection $\cO_{\mathcal{Y}_\eta^\Diamond}\twoheadrightarrow \mathfrak{f}_{\eta *}\cO_{\mathcal{X}_\eta^\Diamond}$. Since $\mathcal{X}_\eta$ is a smooth analytic adic space over $\Spa \breve{E}$, we have an identification 
    \begin{equation*}
        \Gamma(\mathcal{X}_\eta,\cO^+ )\cong \Gamma(\mathcal{X}_\eta^\Diamond,\cO^+)
    \end{equation*}
    and similarly for $\mathcal{Y}$. This in particular means we obtain a surjection $B^+\twoheadrightarrow A^+$, which shows the map $f$ is a closed immersion.
\end{proof}

\begin{remark}
    \phantom{s}
    \begin{enumerate}
        \item We note that the argument in the second part of the proof applies to any type of local Shimura data, as long as the $v$-theoretic integral model $\mathcal{M}^{\mathrm{int}}_{(\mathcal{G}, b_G, \mu_G})$ is represented by a Rapoport--Zink space, i.e. a formal scheme which is formally smooth and formally locally of finite type over $\Spf \cO_{\breve{E}}$.
        \item A rather elementary but crucial point is the $v$-sheaf functor $\mathfrak{X}\to \mathfrak{X}^{\Diamond}$ on formal schemes is in general not fully faithful. Thus, the proposition above relies on the nice geometric properties of the Rapoport--Zink spaces.
    \end{enumerate}
\end{remark}

\begin{corollary}
    Let $(G^{\ad}, [b_{G^{\ad}}], \{\mu_{G^{\ad}}\}) \to (G'^{\ad}, [b_{G'^{\ad}}], \{\mu_{G'^{\ad}}\})$ be a closed embedding of unramified local Shimura data of \textit{adjoint} abelian type. This induces a closed embedding $\RZ_{G^{\ad}} \to \RZ_{G'^{\ad}}$.
\end{corollary}
\begin{proof}
    This immediately follows, since it is a special case of the abelian type.
\end{proof}

For an unramified local Shimura datum of abelian type $(G, [b_G], \{\mu_G\})$, and a compact open subgroup $K_G\subset G(\Q_p)$, we write $\RZ^{K_G}_{G}$ for the local Shimura variety with $K_G$-level structure. 

Let $(H, [b_H], \{\mu_H\})$ be an associated local Shimura datum of Hodge type. Let $K'$ for the image of $K_G$ under the map $G\to G^{\ad}\cong H^{\ad}$. Let $K_H$ be preimage of $K'$ along the quotient map $H\to H^{\ad}$. Then $K_H \subset H(\Q_p)$ is an open compact subgroup, and $\RZ^{K_H}_H$ is a local Shimura variety of Hodge type with $K_H$-level structure, built from the datum $(H, [b_H], \{\mu_H\})$.

When no confusion would arise, we write $K =K_G \subset G(\Q_p)$.

\section{Constructions of auxiliary spaces}\label{sec: aux}\label{Sec: 9}

In this section, we construct a few formal schemes related to the adjoint local Shimura datum $(G^{\ad},[b_{G^{\ad}}],\{\mu_{G^{\ad}}\})$. We should explain why this adjoint datum needs separate constructions, as we already have the Rapoport--Zink space $\RZ_{G^{ad}}$ from \cite{ShenAbelianRZ}. The subtlety comes from our strategy of using the correspondence diagram

\[\begin{tikzcd}
                                                      & \RZ_{P_G}^{\mathrm{rig}} \arrow[ld, "{\pi_{1,G}}"] \arrow[rd, "{\pi_{2,G}}"] &                      \\
\RZ_{L_G}^{\mathrm{rig}} \arrow[ru, "s_G", bend left] &                                                                              & \RZ_G^{\mathrm{rig}}\\
\end{tikzcd}\]
as in \cite[Lemma 4.3.1]{HongPublished} for the datum $(G,[b_G],\{\mu_G\})$ \textit{of abelian type}. This requires us to relate the diagram above to the correspondence diagram for the Hodge type $(H,[b_H],\{\mu_H\})$:

\[\begin{tikzcd}
                                                      & \RZ_{P_H}^{\mathrm{rig}} \arrow[ld, "{\pi_{1,H}}"] \arrow[rd, "{\pi_{2,H}}"] &                      \\
\RZ_{L_H}^{\mathrm{rig}} \arrow[ru, "s_H", bend left] &                                                                              & \RZ_H^{\mathrm{rig}}\\
\end{tikzcd}\]

It is natural to think the connection between the two diagrams is given by an analogous diagram for the adjoint datum. But what does this actually mean? Denote by $P_{G^{\ad}}$ (resp. $L_{G^{\ad}}$) the image of $P_H$ (resp. $L_H$) under the quotient map $H\twoheadrightarrow H^{\ad} \cong G^{\ad}$. But note that in general $P_{G^{\ad}}$ is not the adjoint quotient of $P_H$ (similarly for $L_H$). Thus, existing results such as \cite[Proposition 4.9]{ShenAbelianRZ} do not directly tell us what the Rapoport--Zink space associated to $P_{G^{\ad}}$ or $L_{G^{\ad}}$ is. We will resolve this problem in this subsection.

\begin{lemma} \label{lem:Levi-abelian}
    Let $(G, [b_G], \{\mu_G\})$ be an unramified Hodge--Newton reducible local Shimura datum of abelian type, and let $(L_G, [b_G], \{\mu_{G}\})$ be the local Shimura datum corresponding to a Levi subgroup $L_G \subsetneq G$. Then $(L_G, [b_G], \{\mu_{G}\})$ is an unramified local Shimura datum of abelian type.
\end{lemma}
\begin{proof}
    Let $(H, [b_H], \{\mu_H\})$ be an unramified local Shimura datum of Hodge type associated to $(G, [b_G], \{\mu_G\})$. Write $L'$ for the image of $L$ under the map $G\to G^{\ad}\cong H^{\ad}$. Then let $L_H$ be preimage of $L'$ along the quotient map $H\to H^{\ad}$. 
    Thus, $L_H$ is an (unramified) Levi subgroup of $H$, and $(H,[b_H],\{\mu_H\})$ induces an unramified local Shimura datum $(L_H,[b_H],\{\mu_{H}\})$ of Hodge type. It is then easy to check $(L_H,[b_H],\{\mu_{H}\})$ is associated to $(L_G, [b_G], \{\mu_{G}\})$. \end{proof}

Recall that for an unramified local Shimura datum $(H, b_H, \{\mu_H\})$ of Hodge type, we have a map 
\begin{equation*}
    \omega_H: \RZ_H \to c_{b_H, \mu_H}\pi_1(H)^{\Gamma},
\end{equation*} 
and the fibre over the point $c_{b_H,\mu_H}$ is by definition $\RZ^+_H$. The same applies to a Levi subgroup $L_H$ of $H$, i.e., $\RZ^+_{L_H}:= \omega^{-1}_{L_H}(c_{b_H, \mu_{H}})$ is the fibre of $\omega_{L_H}: \RZ_{L_H} \to c_{b_H, \mu_{H}}\pi_1(L_H)^{\Gamma}$.

Since $\RZ_H$ is of Hodge type, we adopt the construction of the correspondence diagram from \cite[\S 4]{HongPublished}.

From $Q^\vee_{L_H} \subseteq Q^\vee_{H}$, we obtain a surjective homomorphism
\begin{equation*}
    \varphi: \pi_1(L_H) := X_*(T_H)/ Q^\vee_{L_H}\twoheadrightarrow  X_*(T_H)/ Q^\vee_{H} =: \pi_1(H).
\end{equation*}
Note we have $\varphi(c_{b_H, \mu_{H}})= c_{b_H, \mu_H}$.
Let $\varphi^{\Gamma}: \pi_1(L_H)^{\Gamma} \twoheadrightarrow \pi_1(H)^{\Gamma}$ be the map induced by taking $\Gamma$-invariants on the map $\varphi$. We get a map of cosets $c_{b_H, \mu_{H}}\pi_1(L_H)^{\Gamma} \twoheadrightarrow c_{b_H, \mu_H}\pi_1(H)^{\Gamma}$.

We have a natural map $\omega_{P_H}:=\omega_H \circ \pi_{2,H}: \RZ_{P_H} \to c_{b_H, \mu_H}$. We define $\RZ^+_{P_H}:= \omega_{P_H}^{-1}(c_{b_H, \mu_H})$ to be the fibre, and also set $\RZ^+_{P_G}:= \RZ^+_{P_H}$.

By construction, $\RZ_G^+=\RZ_H^+$, $\RZ_{P_G}^+=\RZ_{P_H}^+$ and $\RZ_{L_G}^+=\RZ_{L_H}^+$ are closed formal subschemes of $\RZ_H$, $\RZ_{P_H}$ and $\RZ_{L_H}$ respectively. We extend the construction to the rigid analytic generic fibre to obtain $\RZ_{G}^{+\mathrm{rig}} = \RZ_{H}^{+\mathrm{rig}}$, $\RZ_{P_G}^{+\mathrm{rig}} = \RZ_{P_H}^{+\mathrm{rig}}$ and $\RZ_{L_G}^{+\mathrm{rig}} = \RZ_{L_H}^{+\mathrm{rig}}$ respectively.

Let $\pi_{1,G}^+=\pi_{1,H}^+ := \pi_{1,H}\big|_{\RZ_{H}^{+}}$ and $\pi_{1,L_G}^+=\pi_{1,L_H}^+ := \pi_{1,L_H}\big|_{\RZ_{L_H}^{+}}$. Similarly, we can define $\pi_{2,G}^+= \pi_{2,H}^+ := \pi_{2,H}\big|_{\RZ_{H}^{+}}$. Again, we abuse notation and use the same convention to refer to the corresponding induced maps on the (restricted) rigid analytic fibres.

Let $\beta_G: \pi_1(G)^{\Gamma} \to \pi_1(G^{\ad})^{\Gamma}$. We define the formal scheme $\RZ_{P_G}:= J_{b_{G}}(\Q_p)\RZ_{P_G}^+ = \beta_G^*(\RZ_{P_{G^{\ad}}})$. Let $\RZ_{P_G}^{\rig}$ be the rigid analytic generic fibre.

We add a few more notations:

    \begin{itemize}
        \item $\RZ_{L_G} = J_{b_G}(\Q_p)\RZ_{L_G}^+$, $\;\RZ_{L_G}^{\rig} = J_{b_G}(\Q_p)\RZ_{L_G}^{+\rig}.$ 
        \item $\RZ_{P_G} = J_{b_G}(\Q_p)\RZ_{P_G}^+$, $\;\RZ_{P_G}^{\rig} = J_{b_G}(\Q_p)\RZ_{P_G}^{+\rig}.$ 
    \end{itemize}

\begin{lemma}\label{lem:ab-fibration}
    There exists a diagram
\[\begin{tikzcd}
                                                      & \RZ_{P_G}^{\mathrm{rig}} \arrow[ld, "{\pi_{1,G}}"] \arrow[rd, "{\pi_{2,G}}"] &                      \\
\RZ_{L_G}^{\mathrm{rig}} \arrow[ru, "s_G", bend left] &                                                                              & \RZ_G^{\mathrm{rig}}\\
\end{tikzcd}\]
such that:
\begin{enumerate}
    \item $\pi_{2,G}:\RZ_{P_G}^{\rig} \to \RZ_{G}^{\rig}$ is an isomorphism. 
    \item $s_G$ is a closed immersion.
    \item There is a map $\pi_{1,G}: \RZ_{P_G}^{\rig} \to \RZ_{L_G}^{\rig}$, which is a fibration in balls.
\end{enumerate}
\end{lemma}
\begin{proof}
\begin{enumerate}
    \item 
    Recall that $\pi_{2,H}: \RZ_{P_H} \to \RZ_{H}$ is a local isomorphism which induces the isomorphism $\pi_{2,H}: \RZ_{P_H}^{\mathrm{rig}} \xrightarrow{\cong} \RZ_{H}^{\mathrm{rig}}$ \cite[4.2.4]{HongPublished}. We restrict to obtain the local isomorphism $\pi_{2,G}^+: \RZ_{P_G}^+ \to \RZ_G^+$, which induces $\pi_{2,G}^+: \RZ_{P_G}^{+ \mathrm{rig}} \xrightarrow{\cong} \RZ_G^{+\mathrm{rig}}$. We thus find that $\pi_{2,G}: \RZ_{P_G} = J_{b_G}(\Q_p)\RZ_{P_G}^+ \to J_{b_G}(\Q_p)\RZ_G^+ = \RZ_G$ is also a local isomorphism, thus inducing the isomorphism $\pi_{2,G}:\RZ_{P_G}^{\rig} \xrightarrow{\cong} \RZ_{G}^{\rig}$.

    \item The composition $\pi_{2,G} \circ s_G: \RZ_{L_G}^{\rig} \hookrightarrow \RZ_{G}^{\rig}$ is the closed embedding on the rigid analytic generic fibre that is induced by the closed embedding $(L_G, [b_G], \{\mu_{G}\}) \hookrightarrow (G, [b_G], \{\mu_G\})$ as per \Cref{prop: functoriality-ab}. It follows that $s_G$ is a closed immersion.
    \item We first prove this for the $+$-closed formal subschemes:
    \begin{align*}
        \pi_{1,G}^+(\RZ^{+\rig}_{P_G}) = \pi_{1,H}^+(\RZ^{+\rig}_{P_H}) = \pi_{1,H}^+(\omega_{P_H}^{-1}(c_{b_H, \mu_H})) = \pi_{1,H}^+\circ (\omega_H \circ \pi_{2,H})^{-1}(\varphi(c_{b_H, \mu_{H}})) \\
        = \pi_{1,H}^+\circ \pi_{2,H}^{-1} \circ \omega_H^{-1} \circ \varphi (c_{b_H, \mu_{H}}) \subseteq \omega_{L_H}^{-1}(c_{b_H, \mu_{H}}) = \RZ_{L_H}^{+\rig} = \RZ_{L_G}^{+\rig}
    \end{align*}
    Thus, we get the map $\RZ_{P_G}^{\rig} = J_{b_G}(\Q_p) \pi_{1,G}^+(\RZ^{+\rig}_{P_G}) \xrightarrow{\pi_{1,G}} J_{b_G}(\Q_p) \RZ_{L_G}^{+\rig} = \RZ_{L_G}^{\rig}$.
    
    It now suffices to show that $\pi_{1,G}^+$ is a fibration in balls. This follows from \cite[Lemma 4.3.1]{HongPublished} by restricting to $\RZ_{P_H}^{+\rig} = \RZ_{P_G}^{+\rig}$. 
\end{enumerate}
\end{proof}

\section{Cohomological considerations}\label{Sec: 10}

We again consider the correspondence diagram in \cite[\S 4]{HongPublished} for a Rapoport--Zink space $\RZ_G$ attached to an unramified local Shimura datum $(G,[b_G],\{\mu_G\})$ of abelian type. We write $$\H^i_{\et}(\RZ_{G}^{K_G}):=\H^i_{c}(\RZ_{G}^{K_G}\otimes_{\breve{\Q}_p}\C_p,\Q_{\ell}(\dim \RZ_{G,b}^{K_G}))$$ for any compact open subgroup $K_G$ of $G(\Q_p)$. Let $E^{\mathrm{opt}}$ be the finite extension of $E_G$ constructed in \Cref{thm: wdd-Shen}. We get a representation of $G(\Q_p)\times J_{b_G}(\Q_p)\times \cW_{E_G^{\opt}}$ on $\Q_{\ell}$-vector spaces by setting:
\begin{equation*}
    \H^i(\RZ_G^{\infty}):=\varinjlim_{K_G\subset G(\Q_p)} \H^i_{\et}(\RZ_{G}^{K_G}).
\end{equation*}
For an admissible $\ell$-adic representation $\rho$ of $J_{b_G}(\Q_p)$, we write
\begin{equation*}
    \H^{i,j}(\RZ_{G}^\infty)_\rho := \varinjlim_{K_G \subset G(\Q_p)}\;\Ext^{j}_{J_{b_G}(\Q_p)}(\H^i(\RZ_{G}^{K_G}),\rho).
\end{equation*}
Each $\H^{i,j}(\RZ_{G}^\infty)_\rho$ has a natural action of $G(\Q_p) \x \cW_{E_G^{\opt}}$. We get a virtual representation:
    
\[\H^\bullet(\RZ^\infty_{G})_\rho := \sum_{i,j \geq 0}(-1)^{i+j}\H^{i,j}(\RZ_{G}^\infty)_\rho \\\]
\begin{definition}
    For any compact open subgroup $K_G\subset G(\Q_p)$, we define rigid-analytic space $\RZ_G^{+,K_G}:= \RZ_G^{+}\times_{\RZ_G}\RZ_G^{K_G}$. This construction gives a projective system $\{\RZ_G^{+,K_G}\}_{K_G\subset G(\Q_p)}$ of rigid-analytic spaces over $\breve{\Q}_p$, where the transition maps are finite \'etale.
\end{definition}

 We now introduce cohomological analogues for the formal subschemes:
 \begin{enumerate}
     \item \[ \H^i(\RZ_G^+):=\varinjlim_{K_G\subset G(\Q_p)} \H^i_{\et}(\RZ_{G}^{+, K_G})\]
     \item \[\H^{i,j}(\RZ_{G}^{+,\infty})_\rho := \varinjlim_{K_G \subset G(\Q_p)}\;\Ext^{j}_{J_b(\Q_p)}(\H^i(\RZ_{G}^{+, K_G}),\rho)\]
     \item \[\H^\bullet(\RZ^{+ , \infty}_{G})_\rho := \sum_{i,j \geq 0}(-1)^{i+j}\H^{i,j}(\RZ_{G}^{+,  \infty})_\rho \]
 \end{enumerate}
   
Similarly, we can define $\H^\bullet(\RZ^{+,\infty}_{L_G})_\rho$, as well as $\H^\bullet(\RZ^\infty_{P_G})_\rho$ and $\H^\bullet(\RZ^{+,\infty}_{P_G})_\rho$.

\subsection{Formal subschemes}

We have the following result from \cite{HongPublished}:

For every integer $m>0$, there exists a formal scheme $\psi_H:\RZ_{P_H}^{(m)} \to \RZ_{P_H}$ with the following properties:
\begin{enumerate}
    \item For any $R \in \mathrm{Nilp}_{\breve{\Z}_p}$, a morphism $f: \spf\;(R) \to \RZ_{P_H}$ factors through $\RZ_{P_H}^{(m)}$ if and only if the filtration $f^*\cX^{\bullet}_{P_H}[p^m]$ splits.
    \item $\RZ_{P_H}$ and $\RZ_{P_H}^{(m)}$ are isomorphic as $\RZ_{L_H}$-formal schemes, via $\pi_{1,H}: \RZ_{P_H} \to \RZ_{L_H}$.
\end{enumerate}

Let $K_{H,(m)}' := \ker(P_H(\Z_p) \to P_H(\Z_p/p^m\Z_p))$. Let $\cP_{H,m}$ and $\cP'_{H,m}$ be two distinct covers of $\RZ_{P_H}^{(m)}$, defined by two Cartesian diagrams:

\begin{center}
   \begin{tikzcd}
{\cP_{H,m}} \arrow[r] \arrow[d]     & {\RZ_{P_H}^{\rig, (m)}} \arrow[d] &  & {\cP'_{H,m}} \arrow[d] \arrow[r]  & {\RZ_{P_H}^{K_{H, (m)}'}} \arrow[d, "{\pi_{1,H}}"] \\
{\RZ_{P_H}^{K_{H, (m)}'}} \arrow[r] & \RZ_{P_H}^{\rig}                  &  & {\RZ_{L_H}^{\rig, (m)}} \arrow[r] & \RZ_{L_H}^{\rig}                                  
\end{tikzcd}
\end{center}

Let $D_H = \dim(\RZ_{P_H}) - \dim(\RZ_{L_H})$. By \cite[Prop 4.3.2]{HongPublished}, we have the following quasi-isomorphism for every integer $m>0$:

\[R\Gamma_c(\RZ_{P_H}^{K_{H, (m)}'} \otimes_{\breve{\Q}_p} \C_p, \overline{\Q}_\ell) \cong R\Gamma_c(\RZ_{L_H}^{K_{H, (m)}'}\otimes_{\breve{\Q}_p} \C_p, \overline{\Q}_\ell(-D_H))[-2D_H] \]
This yields 
\begin{align*}
    \mathrm{H}^\bullet(\RZ_{L_{H}}^{\infty})_\rho \cong \mathrm{H}^\bullet(\RZ_{P_{H}}^{\infty})_\rho.
    \end{align*}

Recall that we set $\RZ_{L_G}^+ := \RZ_{L_H}^+$ (thus giving rise to $\RZ_{L_G}^{+, (m)}$) and $\RZ_{P_G}^+:= \RZ_{P_H}^+$. We define a formal subscheme $\RZ_{P_G}^{+, (m)} := \psi_H^{-1}\circ \omega_{P_H}^{-1}(c_{b_H, \mu_H})$, equipped with the map $\psi_G^+: \RZ_{P_G}^{+, (m)} \to \RZ_{P_G}^+$. 

The following corollary is a consequence of \Cref{lem:ab-fibration}. 
\begin{proposition} \label{prop: + level Levi and parabolic}
    There is an isomorphism between the following virtual $\ell$-adic representations of $P_G(\Q_p) \times \cW_{E_G^{\opt}}$:
    \begin{align*}
    \mathrm{H}^\bullet(\RZ_{L_{G}}^{+,\infty})_\rho \cong \mathrm{H}^\bullet(\RZ_{P_{G}}^{+,\infty})_\rho
    \end{align*}
\end{proposition}

\begin{proof}

We get natural diagrams:

\begin{center}
   \begin{tikzcd}
{\cP^+_{G,m}} \arrow[r] \arrow[d]     & {\RZ_{P_G}^{+\rig, (m)}} \arrow[d] &  & {\cP'^+_{G,m}} \arrow[d] \arrow[r]  & {\RZ_{P_G}^{+K_{G, (m)}'}} \arrow[d, "{\pi^+_{1,G}}"] \\
{\RZ_{P_G}^{+K_{G, (m)}'}} \arrow[r] & \RZ_{P_G}^{+\rig}                  &  & {\RZ_{L_G}^{+\rig, (m)}} \arrow[r] & \RZ_{L_G}^{+\rig}                                  
\end{tikzcd}
\end{center}

Let $D^+_{G} = \dim(\RZ^+_{P_G}) - \dim(\RZ^+_{L_G})$. It follows that we have the following quasi-isomorphism for every integer $m>0$:

\[R\Gamma_c(\RZ_{P_G}^{+K_{H, (m)}'} \otimes_{\breve{\Q}_p} \C_p, \overline{\Q}_\ell) \cong R\Gamma_c(\RZ_{L_G}^{+K_{H, (m)}'}\otimes_{\breve{\Q}_p} \C_p, \overline{\Q}_\ell(-D^+_{G}))[-2D^+_{G}] \]

This yields 
\begin{align*}
    \mathrm{H}^\bullet(\RZ_{L_{G}}^{+,\infty})_\rho \cong \mathrm{H}^\bullet(\RZ_{P_{G}}^{+,\infty})_\rho
\end{align*}
\end{proof}

\begin{corollary}\label{cor: Levi and parabolic abelian} 
    There is an isomorphism between the following virtual $\ell$-adic representations of $P_G(\Q_p) \times \cW_{E_G^{\opt}}$:
    \begin{align*}
    \mathrm{H}^\bullet(\RZ_{L_{G}}^{\infty})_\rho \cong \mathrm{H}^\bullet(\RZ_{P_{G}}^{\infty})_\rho
    \end{align*}
\end{corollary}

\begin{proof}

    We have Cartesian diagrams:

\begin{center}
    \begin{tikzcd}
{J_{b_G}(\Q_p)\cP^+_{G,m}} \arrow[r] \arrow[d]    & {\RZ_{P_G}^{\rig, (m)}} \arrow[d] &  & {J_{b_G}(\Q_p)\cP'^+_{G,m}} \arrow[d] \arrow[r]     & {\RZ_{P_G}^{K_{G, (m)}'}} \arrow[d, "{\pi_{1,G}}"] \\
{\RZ_{P_G}^{K_{G, (m)}'}} \arrow[r] & \RZ_{P_G}^{\rig}                  &  & {\RZ_{L_G}^{\rig, (m)}} \arrow[r] & \RZ_{L_G}^{\rig}                    \end{tikzcd}
\end{center}

Here, $\;\RZ_{P_G}^{K_{G, (m)}'} = J_{b_G}(\Q_p)\RZ_{P_G}^{+,K_{G, (m)}'}$. In particular, we get the following isomorphism:

\begin{align*}
   \mathrm{H}^\bullet(\RZ_{P_{G}}^{\infty})_\rho \cong \mathrm{H}^\bullet(\RZ_{L_{G}}^{\infty})_\rho
\end{align*}

\end{proof}

\subsection{Adjoint level}

We note that $\RZ_{L_{H^{\ad}}}: = \RZ_{L_H}/X_*(Z_H)^{\Gamma}$ is distinct from but related to the formal space $\RZ_{(L_H)^{\ad}} = \RZ_{L_H}/X_*(Z_{L_H})$. 

\begin{lemma} Suppose the maps $\beta_H: \pi_1(H)^{\Gamma} \to \pi_1(H^{\ad})^{\Gamma}$ and $\beta_G: \pi_1(G)^{\Gamma} \to \pi_1(G^{\ad})^{\Gamma}$ are surjective.
    \begin{align*}
        \RZ_{L_{G^{\ad}}} \cong \RZ_{L_{H}^{\ad}} \x_{\pi_1((L_H)^{\ad})} \pi_1(L_{H^{\ad}})^{\Gamma}
    \end{align*}
\end{lemma}

\begin{proof}
Recall that $L_{H^{\ad}}$ is the image of $L_H$ under the projection $H \to H^{\ad}$, and can be identified with $L_H/Z_H$. By the third isomorphism theorem, we see:
\begin{align*}
    (L_H)^{\ad} = L_H/Z_{L_H} \cong (L_H/Z_H) / (Z_{L_H}/Z_H) = L_{H^{\ad}} / (Z_{L_H}/Z_H)
\end{align*}
In other words, there is a short exact sequence:
\begin{align*}
    0 \longrightarrow Z_{L_H}/Z_H \longrightarrow L_{H^{\ad}} \longrightarrow (L_H)^{\ad} \longrightarrow 0
\end{align*}
By \cite{Borovoi1989preprint}, we obtain the following short exact sequence by applying $\pi_1(-)$ to the sequence above:
\begin{align*}
    0 \longrightarrow X_*(Z_{L_H})/ X_*(Z_H) \longrightarrow \pi_1(L_{H^{\ad}}) \longrightarrow \pi_1((L_H)^{\ad}) \longrightarrow 0
\end{align*}
Taking the $\Gamma$-invariant subgroups is exact by the vanishing of $\H^1(\Gamma, \pi_1(Z_H))$ as in \Cref{rmk: LES vanishing}, giving us:
\begin{align*}
    0 \longrightarrow X_*(Z_{L_H})^\Gamma/ X_*(Z_H)^\Gamma \longrightarrow \pi_1(L_{H^{\ad}})^\Gamma \longrightarrow \pi_1((L_H)^{\ad})^\Gamma \longrightarrow 0
\end{align*}

Now, we consider a group-theoretic result. Let $\sX = \RZ_{L_H}$, $H_1 = X_*(Z_{L_H})^\Gamma$, $H_2 = X_*(Z_{H})^\Gamma$, $G_1 = \pi_1(L_{H^{\ad}})^\Gamma$ and $G_2 = \pi_1((L_H)^{\ad})^\Gamma $. Thus, we find:
\begin{align*}
    \RZ_{L_{G^{\ad}}} = \RZ_{L_{H^{\ad}}} \cong \RZ_{(L_{H})^{\ad}} \x_{\pi_1((L_H)^{\ad})} \pi_1(L_{H^{\ad}})^{\Gamma}.
\end{align*}

\end{proof}
\begin{lemma}\label{lem:formal-scheme-quotient}
    Let $\sX$ be a formal scheme over $W$ with actions of the group $W$-schemes $H_1$, and let $H_2$ be a subgroup of $H_1$. Suppose there is a short exact sequence
\begin{align*}
    0 \longrightarrow H_1/H_2 \longrightarrow G_1 \longrightarrow G_2 \longrightarrow 0
\end{align*}
 of group schemes over $W$. Then there is a natural isomorphism of formal schemes over $W$ 
\begin{align*}
    (\sX/H_1)\x_{G_2} G_1 \cong \sX/H_2
\end{align*}
\end{lemma}
\begin{proof}
    Let $G$ be a group scheme acting on the formal scheme $\sX$. 
    By \cite[\href{https://stacks.math.columbia.edu/tag/0AIR}{Tag 0AIR}]{stacks-project}, the fppf site of a formal scheme is subcanonical. This implies the \'etale site is subcanonical. So we may regard $\sX$, $H_i$, and $G_i$ ($i=1,2$) as \'etale sheaves over $\Spf W$.
    
    We make use of the following group-theoretic fact:
    Suppose there is a set $S$ with the action of (abstract) groups $H_1$ and $H_2$. If there is a short exact sequence of groups
\begin{align*}
    0 \longrightarrow H_1/H_2 \longrightarrow G_1 \longrightarrow G_2 \longrightarrow 0
\end{align*}
Then 
\begin{align*}
    (S/H_1)\x_{G_2} G_1 \cong S/H_2
\end{align*}

In particular, isomorphism $(\sX/H_1)\x_{G_2} G_1 \cong \sX/H_2$ holds since it holds over $R$-points for any $W$-algebra $R$. \end{proof}

\Cref{lem:formal-scheme-quotient} establishes the following projection:
\begin{align*}
    \RZ_{L_{G^{\ad}}} = \RZ_{L_{H^{\ad}}}  \cong \RZ_{(L_{H})^{\ad}} \x_{\pi_1((L_H)^{\ad})} \pi_1(L_{H^{\ad}})^{\Gamma} \xrightarrow{p_1} \RZ_{(L_{H})^{\ad}} = \RZ_{(L_{G})^{\ad}}.
\end{align*}
Let $\RZ_{P_{G^{\ad}}}: = \RZ_{P_H}/X_*(Z_H)^{\Gamma}$. This allows us to have the following diagram: \\
\begin{center}
    \begin{tikzcd}
\RZ_{L_H}^{\mathrm{rig}} \arrow[d, "/X_*(Z_H)^{\Gamma}"] & \RZ_{P_H}^{\mathrm{rig}} \arrow[l, "{\pi_{1,H}}"'] \arrow[d, "/X_*(Z_H)^{\Gamma}"] \arrow[rr, "{\pi_{2,H}}"] \arrow[rr, "\cong"']                               &                              & \RZ_H^{\mathrm{rig}} \arrow[ld] \\
\RZ_{L_{G^{\ad}}}^{\mathrm{rig}} \arrow[d, "p_1"']       & \RZ_{P_{G^{\ad}}}^{\mathrm{rig}} \arrow[r, "\cong"'] \arrow[ld, "{p_1 \circ \pi_{1, G^{\ad}}}"] \arrow[r, "{\pi_{2,G^{\ad}}}"] \arrow[l, "{\pi_{1, G^{\ad}}}"'] & \RZ_{G^{\ad}}^{\mathrm{rig}} &                                 \\
\RZ_{(L_{G})^{\ad}}^{\mathrm{rig}}                         &                                                                                                                                                                 &                              &                                
\end{tikzcd} \\

\end{center}

In particular, since $\pi_{1,H}$ is a fibration of balls, so is $\pi_{1,G^{\ad}}$. \\

Let $(K_H)^{\ad}$ be the image of $K_H$ under the map $H \to G^{\ad}$. Let $(K_H')^{\ad} = (K_H)^{\ad} \cap P_{G^{\ad}}(\Z_p)$. We use this construction to define
\begin{align*}
    \H^i(\RZ_{P_{G^{\ad}}}^{(K_H')^{\ad}}) := \H^i(\RZ_{P_{H}}^{K_H'}/X_*(Z_H)). \\
    \H^i(\RZ_{P_{G^{\ad}}}^\infty) := \varinjlim_{(K_{H}')^{\ad}}\H^i(\RZ_{P_{G^{\ad}}}^{(K_H')^{\ad}}).
\end{align*}

\begin{theorem}\label{thm:HV adjoint}
Let $(G, [b_G], \{\mu_G\})$ be an unramified non-basic local Shimura datum of abelian type, which is Hodge--Newton reducible with respect to a fixed parabolic subgroup $P_G$ and Levi factor $L_G$. Let $E^{\mathrm{opt}}$ be the finite extension of $E_G$ constructed in \Cref{thm: wdd-Shen}. Then for any admissible $\overline{\Q}_\ell$-representation $\rho$ of $J_{b_{G^{\ad}}}(\Q_p)$, we have an equality
    \begin{align*}
        \mathrm{H}^\bullet(\RZ^\infty_{G^{\ad}})_\rho = \mathrm{Ind}_{P_{G^{\ad}}(\Q_p)}^{G^{\ad}(\Q_p)}\mathrm{H}^\bullet(\RZ^\infty_{L_{G^{\ad}}})_\rho. 
    \end{align*}
of virtual representations of $G^{\ad}(\Q_p)\times \cW_{E^{\opt}}$.
\end{theorem}
\begin{proof}
By \cite[Proposition~4.3.2]{HongPublished}, there is an equality
\begin{align*}
    \mathrm{H}^\bullet(\RZ_{L_H}^\infty)_\rho = \mathrm{H}^\bullet(\RZ_{P_H}^\infty)_\rho.
\end{align*}
of virtual representations of $P_H(\Q_p)\times\cW_{E_H}$. Recall that by definition, $H^{\ad} = G^{\ad}$, $L_{H^{\ad}} = L_{G^{\ad}}$ and $P_{H^{\ad}} = P_{G^{\ad}}$.
Also note for any level $K$, $$R\Hom_{J_{b_G}(\Q_p)}(R\Gamma_c(\RZ_{L_H}^K/X_*(Z_H)^{\Gamma}),\rho)\simeq R\Hom_{J_{b_G}(\Q_p)}(R\Gamma_c(\RZ_{L_H}^K),\rho)_{hX_*(Z_H)^{\Gamma}}$$
Unwinding the definition of $\mathrm{H}^\bullet(\RZ_{L_{G^{\ad}}}^\infty)_\rho$, we then get
\begin{align*}
    \mathrm{H}^\bullet(\RZ_{L_{G^{\ad}}}^\infty)_\rho = \mathrm{H}^\bullet((\RZ_{L_{H}}/X_*(Z_H)^{\Gamma})^\infty)_\rho = \mathrm{H}^\bullet((\RZ_{P_{H}}/X_*(Z_H)^{\Gamma})^\infty)_\rho = \mathrm{H}^\bullet(\RZ_{P_{G^{\ad}}}^\infty)_\rho.
\end{align*}
Now, by \cite[Proposition~4.3.3]{HongPublished}, we know that :
\begin{align*}
\RZ_{H}^{K_H}= \bigsqcup_{K_H \backslash H(\Q_p) / P_H(\Q_p)} \RZ_{P_{H}}.
\end{align*}

Observe that by definition, $K_H/Z_H \cong K_G/Z_G$. Thus, by the third isomorphism theorem:
\begin{align*}
     K_{H}\backslash H(\Q_p)/P_H(\Q_p) \;&\cong \;(K_{H}/Z_H)\backslash H^{\ad}(\Q_p)/P_{H^{\ad}}(\Q_p) \\
    &\cong \; (K_{G}/Z_G)\backslash G^{\ad}(\Q_p)/P_{G^{\ad}}(\Q_p) \\
    &\cong \; K_{G^{\ad}}\backslash G^{ad}(\Q_p)/P_{G^{\ad}}(\Q_p).
\end{align*}

We use this in the following manner:
\begin{align*}
    \RZ_{G^{\ad}}^{K_{G^{ad}}} =  \RZ_H^{K_H}/X_*(Z_H)^\Gamma &= \bigsqcup_{K_H \backslash H(\Q_p) / P_H(\Q_p)} \RZ_{P_{H}}^{K_H}/X_*(Z_H)^\Gamma \\ &\cong \; \bigsqcup_{K_{G^{\ad}}\backslash G^{\ad}(\Q_p)/P_{G^{\ad}}(\Q_p)} \RZ_{P_{G^{\ad}}}^{K_{G^{ad}}}.
\end{align*}

It follows that
\begin{align*}
    \H^\bullet(\RZ_{G^{\ad}}) = \mathrm{Ind}_{P_{G^{\ad}}(\Q_p)}^{G^{\ad}(\Q_p)}\mathrm{H}^\bullet(\RZ^\infty_{P_{G^{\ad}}})_\rho.
\end{align*}

\end{proof}

\begin{convention}\label{conv: associated-open-compacts}
Let $(G,[b_G],\{\mu_G\})$ be an unramified local Shimura datum of abelian type and $(H,[b_H],\{\mu_H\})$ an associated Hodge type. We fix an isomorphism $$(G^{\ad}, [b_{G^{\ad}}], \{\mu_{G^{\ad}}\}) \xrightarrow{\sim} (H^{\ad}, [b_{H^{\ad}}], \{\mu_{H^{\ad}}\})$$ of local Shimura data. For a compact open $K_G\subset G(\Q_p)$, we define its \textit{associated} subgroup $K_H$ of $H(\Q_p)$ in the following way: Let $K_{H^\ad}$ be the image of $K_G$ under the composition $$G(\Q_p)\to G^{\ad}(\Q_p)\xrightarrow{\sim} H^{\ad}(\Q_p).$$Then define $K_H$ as the preimage of $K_{H^{\ad}}$ along the map $H(\Q_p)\to H^{\ad}(\Q_p)$.
\end{convention}

We combine our prior results to deduce the main theorem \Cref{thm: HV full abelian} for the abelian case. 

\begin{proof}[Proof of \Cref{thm: HV full abelian}]
Let $K_G \subset G(\Z_p)$ and $K_H \subset H(\Z_p)$ be associated compact open subgroups as defined in \Cref{conv: associated-open-compacts}. Recall that we have the decomposition
\begin{align*}
    \RZ_H^{K_{H}} = \bigsqcup_{K_{H}\backslash H(\Q_p)/P_H(\Q_p)}\RZ_{P_H}^{K_{H}}
\end{align*}
Restricting it to the $+$-components, we have
\begin{align*}
    \RZ_G^{K_{G},+}:=\RZ_H^{K_{H},+} = \bigsqcup_{K_{H}\backslash H(\Q_p)/P_H(\Q_p)}\RZ_{P_H}^{K_{H},+}
\end{align*}

Since $K_H/(Z_H\cap K_H) \cong K_G/(Z_G\cap K_G)$ by definition, we have the following chain of isomorphisms:
\begin{align*}
    K_{H}\backslash H(\Q_p)/P_H(\Q_p) &\cong \;(K_{H}/Z_H\cap K_H)\backslash H^{\ad}(\Q_p)/P_{H^{\ad}}(\Q_p) \\
    &\cong \; (K_{G}/Z_G\cap K_G)\backslash G^{\ad}(\Q_p)/P_{G^{\ad}}(\Q_p)\\  &\cong \; K_{G}\backslash G(\Q_p)/P_G(\Q_p)
\end{align*}
Hence we can write 
\begin{align*}
    \RZ_G^{K_{G},+} = \bigsqcup_{K_{G}\backslash G(\Q_p)/P_G(\Q_p)}\RZ_{P_H}^{K_{G},+}
\end{align*}
It follows that
\begin{align*}
    \RZ_G^{K_G} \cong J_{b_G}(\Q_p)\RZ_G^{K_G,+}  &= J_{b_G}(\Q_p)\bigsqcup_{K_G\backslash G(\Q_p)/P_G(\Q_p)}\RZ_{P_G}^{K_{G},+} \\
    &= \bigsqcup_{J_{b_G}(\Q_p)/J_{b_G}(\Q_p)^+} \bigsqcup_{K_G\backslash G(\Q_p)/P_G(\Q_p)}\RZ_{P_G}^{K_{G},+}
\end{align*}
Re-ordering the last term gives
\begin{align*}
    \RZ_G^{K_G} = \bigsqcup_{K_G\backslash G(\Q_p)/P_G(\Q_p)}\bigsqcup_{J_{b_G}(\Q_p)/J_{b_G}(\Q_p)^+}\RZ_{P_G}^{K_{G},+} =\bigsqcup_{K_G\backslash G(\Q_p)/P_G(\Q_p)} \RZ_{P_G}^{K_G},
\end{align*}
which, by a similar argument as in the proof of \cite[Prop 4.3.3]{HongPublished}, then implies:
\begin{align*}
     \mathrm{H}^\bullet(\RZ_{G}^\infty)_\rho = \mathrm{Ind}_{P_{G}(\Q_p)}^{G(\Q_p)} \mathrm{H}^\bullet(\RZ_{P_{G}}^\infty)_\rho
\end{align*}
Let $L = L_G$ and $P=P_G$. Combining this identity with \Cref{cor: Levi and parabolic abelian} gives the desired result:
\begin{align*}
    \mathrm{H}^\bullet(\RZ_{G}^\infty)_\rho = \mathrm{Ind}_{P(\Q_p)}^{G(\Q_p)} \mathrm{H}^\bullet(\RZ_{L}^\infty)_\rho
\end{align*}
\end{proof}

\bibliographystyle{amsalpha}
\bibliography{HarrisViehmann}

\end{document}